\documentclass[12pt]{amsart}
\usepackage{a4wide}
\usepackage[utf8]{inputenc}
\usepackage{amsmath}
\usepackage{amsfonts}
\usepackage{amssymb}
\usepackage{amsthm}
\usepackage{subfigure}
\newtheorem{theo}{Theorem}[section]
\newtheorem{prop}[theo]{Proposition}
\newtheorem{coro}[theo]{Corollary}
\newtheorem{lemm}[theo]{Lemma}

\theoremstyle{definition}

\theoremstyle{remark}
\newtheorem{rema}[theo]{Remark}
\newcommand{\Op}{\operatorname{Op}}

\newcommand{\nwc}{\newcommand}
\nwc{\eps}{\epsilon}
\nwc{\ep}{\epsilon}
\nwc{\vareps}{\varepsilon}
\nwc{\Oph}{\operatorname{Op}_\hbar}
\nwc{\la}{\langle}
\nwc{\ra}{\rangle}

\nwc{\mf}{\mathbf} 
\nwc{\blds}{\boldsymbol} 
\nwc{\ml}{\mathcal} 

\nwc{\defeq}{\stackrel{\rm{def}}{=}}

\nwc{\cE}{\ml{E}}
\nwc{\cN}{\ml{N}}
\nwc{\cO}{\ml{O}}
\nwc{\cP}{\ml{P}}
\nwc{\cU}{\ml{U}}
\nwc{\cV}{\ml{V}}
\nwc{\cW}{\ml{W}}
\nwc{\tU}{\widetilde{U}}
\nwc{\IN}{\mathbb{N}}
\nwc{\IR}{\mathbb{R}}
\nwc{\IZ}{\mathbb{Z}}
\nwc{\IC}{\mathbb{C}}
\nwc{\IT}{\mathbb{T}}
\nwc{\IS}{\mathbb{S}}
\nwc{\tP}{\widetilde{P}}
\nwc{\tPi}{\widetilde{\Pi}}
\nwc{\tV}{\widetilde{V}}
\nwc{\supp}{\operatorname{supp}}
\nwc{\rest}{\restriction}

\begin{document}

\title[Concentration and non-concentration on Zoll manifolds]{Concentration and non-concentration for the Schr\"odinger evolution on Zoll manifolds}

\author[Fabricio Maci\`a]{Fabricio Maci\`a}
\author[Gabriel Rivi\`ere]{Gabriel Rivi\`ere}

\address{Universidad Polit\'ecnica de Madrid. DCAIN, ETSI Navales. Avda. Arco de la Victoria s/n. 28040 Madrid, Spain}
\email{Fabricio.Macia@upm.es}
\thanks{FM takes part into the visiting faculty program of ICMAT and is
partially supported by grants ERC Starting Grant 277778 and
MTM2013-41780-P (MEC)}

\address{Laboratoire Paul Painlev\'e (U.M.R. CNRS 8524), U.F.R. de Math\'ematiques, Universit\'e Lille 1, 59655 Villeneuve d'Ascq Cedex, France}
\email{gabriel.riviere@math.univ-lille1.fr}
\thanks{GR is partially supported by the Agence Nationale de la Recherche through the
Labex CEMPI (ANR-11-LABX-0007-01) and the ANR project GeRaSic (ANR-13-BS01-
0007-01)}

\begin{abstract}
We study the long time dynamics of the Schr\"odinger equation on Zoll manifolds. We establish criteria under which the solutions of the Schr\"odinger equation can or cannot concentrate on a given closed geodesic. As an application, we derive some results on the set of semiclassical measures for eigenfunctions of Schr\"odinger operators: we prove that adding a potential to the Laplacian on the sphere results on the the existence of geodesics $\gamma$ such that $\delta_\gamma$ cannot be obtain as a semiclassical measure for some sequence of eigenfunctions. We also show that the same phenomenon occurs for the free Laplacian on certain Zoll surfaces. 
\end{abstract}

\maketitle

\section{Introduction}

In this article we are interested in understanding the dynamics of Schrödinger equations and the structure of eigenfunctions of Schrödinger operators on Zoll manifolds. Recall that a \emph{Zoll manifold} is a smooth, connected, compact, Riemannian manifold without boundary $(M,g)$ such that \emph{all its geodesics are closed}. This means that, for every $x$ in $M$, all the geodesics issued from $x$ come back to $x$. Thanks to Theorem of Wadsley (see \cite{Be78} -- section 7.B), the geodesic flow $\varphi^s$ acting on the unit cotangent bundle $S^*M$ of such a manifold is periodic, meaning that all its trajectories have a minimal common period $l>0$.\footnote{Note that this does not mean that all the geodesics have length equal to~$l$. There may exist exceptional geodesics whose length is strictly smaller than $l$, the Lens spaces (quotients of $\IS^{2m-1}$ by certain finite cyclic groups of isometries) provide an example of this -- see Ex.~$2.43$ in~\cite{Hat02}.} Using the terminology of~\cite{Be78} -- chapter $7$, we will say that the metric $g$ is a $P_l$-metric, or that $(M,g)$ is a $P_l$-manifold. Similarly, in the case where all the geodesics have the same length $l$, we will say that $g$ is a $C_l$-metric or that $(M,g)$ is a $C_l$-manifold. The main examples of $C_l$-manifolds are the compact rank one symmetric spaces (that we will sometimes abbreviate by CROSS) and certain surfaces of revolution -- see chapters~3 and 4 in~\cite{Be78}. It is also known that the set of  $C_{l}$-metrics on $\IS^2$ that can be obtained as smooth deformations of the canonical metric is infinite dimensional and contains many metrics that are not of revolution. The characterization of such deformations is a remarkable result by Guillemin \cite{Gu76}.

Our goal here is to understand the long time dynamics of the following Schr\"odinger equation:
\begin{equation}\label{e:nonsemiclassical-schr}
i\partial_tv(t,x)=\left(-\frac{1}{2}\Delta_g+V(x)\right)v(t,x),\quad v|_{t=0}=u\in L^2(M),
\end{equation}
or the behavior of eigenfunctions:
\begin{equation}\label{e:eigenf}
\left(-\frac{1}{2}\Delta_g+V(x)\right)u(x)=\lambda^2 u(x), \quad ||u||_{L^2(M)}=1,
\end{equation}
in the high-frequency limit $\lambda \rightarrow \infty$.
As usual, $\Delta_g$ is the Laplace Beltrami operator induced by the Riemannian metric $g$ on $M$, and we shall assume that $V$ is in $\ml{C}^{\infty}(M,\IR)$.

The study of the spectral properties of the operator $-\frac{1}{2}\Delta_g+V$ in this geometric context is a problem which has a long history in microlocal analysis starting with the works of Duistermaat-Guillemin~\cite{DuGu75, Gu78a, Gu78b}, Weinstein~\cite{We77} and Colin de Verdi\`ere~\cite{CdV79}. Many other important results on the fine structure of the spectrum of Zoll manifolds were obtained both in the microlocal framework~\cite{Gu81, Ur85, UrZe93, Ze96, Ze97}, and in the semiclassical setting~\cite{Ch80, HeRo84, HaHiSj14} -- see also~\cite{Hi02, HiSj04} in the nonselfadjoint setting.

In this this article we use similar techniques to those that were originally developed in order to study of the spectrum of Schrödinger operators to actually provide some results on the long-time dynamics of the Schr\"odinger evolution~\eqref{e:nonsemiclassical-schr} and the structure of high-frequency eigenfunctions~\eqref{e:eigenf}.

\subsection{Concentration and non-concentration of eigenfunctions}
Let us start by describing our results in the context of eigenfunctions of Schrödinger operators, as they are somewhat simpler to state. We are mainly interested in analyzing how the mass $|u|^2$ of a high-frequency eigenfunction satisfying \eqref{e:eigenf} distributes over $M$. More precisely, consider the set $\mathcal{N}(\infty)$ of probability measures in $M$ that are obtained as follows. A probability measure $\nu$ belongs to $\ml{N}(\infty)$ provided there exist a sequence of eigenfunctions $(u_n)$ :
$$-\frac{1}{2}\Delta_g u_n+Vu_n=\lambda_n^2 u_n, \quad ||u_n||_{L^2(M)}=1,$$
with eigenvalues satisfying $\lambda_n\rightarrow\infty$ such that 
$$\lim_{n\rightarrow\infty}\int_M a(x)|u_n|^2(x)d\operatorname{vol}_g(x)=\int_M a(x)\nu(dx),\quad \text{for every }a\in\ml{C}(M).$$
Measures in $\ml{N}(\infty)$ therefore describe the asymptotic mass distribution sequences of eigenfunctions $(u_n)$ whose corresponding eigenvalues tend to infinity. The problem of characterizing the probability measures in $\mathcal{N}(\infty)$ has attracted a lot of attention in the last forty years. 

In the case of Zoll manifolds, it is well known that $\ml{N}(\infty)$ is contained in $\ml{N}_g$ which is, by definition, the closed convex hull (with respect to the weak-$\star$ topology) of the set of probability measures $\delta_\gamma$, where $\gamma$ is a geodesic of $(M,g)$. Recall that
$$\int_M a(x)\delta_\gamma(dx) = \frac{1}{l}\int_0^l a(\gamma(s))ds,$$
where $l$ denotes the length of $\gamma$ and the parametrization $\gamma(s)$ has unit speed. 

In the case of $\IS^2$ endowed with its canonical metric, it was proved that a ``generic'' orthonormal basis of eigenfunctions satisfies the quantum unique ergodicity property~\cite{Ze94, VdK97}. In particular, in this case, the normalized Riemannian volume is the only accumulation point of the sequence of eigenfunctions under consideration. Note that quantum ergodicity properties were also proved recently for sequences of eigenfunctions on $\IS^d$ satisfying certain symmetry assumptions~\cite{KR14, BrLeLi15}. Yet, there are in fact some situations for which one has
$$\mathcal{N}(\infty)=\ml{N}_g,$$
which is in a certain sense the opposite situation to quantum ergodicity. Jakobson and Zelditch proved in \cite{JaZe99} that this holds when $(M,g)$ is the sphere $\IS^d$ equipped with its canonical metric $g=\text{can}$ and the potential $V$ vanishes identically. This was also shown to hold for general Compact Rank-One Symmetric Spaces (see \cite{Ma08}) and for any manifold of positive constant curvature (see \cite{AzM10} --  this analysis relies on the study of eigenfunctions of the canonical Laplacian on the sphere that are invariant by certain groups of isometries); in both cases one has to assume that $V=0$. 
To the authors' knowledge, the question of whether this is always the case when the potential $V$ does not vanish identically or when $(M,g)$ is a Zoll manifold that is not isometric to a CROSS remains open. 

Here we answer these questions  by the negative. It turns out that these regimes are somewhat intermediate between the quantum ergodicity results and the results on a CROSS. Let us first introduce some notations. Denote by $\mathring{T}^*M$ the cotangent bundle of $M$ with the zero section removed. The \emph{Radon transform} of $V$ is defined as
\begin{equation}\label{e:radon}
\ml{I}_g(V)(x,\xi):=\frac{\|\xi\|_x}{l}\int_0^{\frac{l}{\|\xi\|_x}}V\circ\pi\circ
\varphi^{\tau}(x,\xi)d\tau,\quad (x,\xi)\in \mathring{T}^*M.
\end{equation}
Above, $\pi:T^*M \longrightarrow M$ stands for the canonical projection, $\varphi^\tau$ is the geodesic flow of $(M,g)$ acting on the cotangent bundle, and $l$ is the minimal common period of the trajectories of $\varphi^\tau|_{S^*M}$. Clearly, $\ml{I}_g(V)\in\ml{C}^\infty(\mathring{T}^*M)$ is zero-homogeneous with respect to the variable $\xi$. We shall consider the projection onto $M$ of the set of critical points of the Radon transform of $V$:
$$\ml{C}(V)=\{x\in M : d_{(x,\xi)} \ml{I}_g(V) = 0 \text{ for some } \xi\in T^*_x M\setminus \{0\}\}.$$
The set $\ml{C}(V)$ is a union of geodesics of $(M,g)$ and if the critical points of $\ml{I}_g(V)$ are non degenerate in an appropriate sense, $\ml{C}(V)$  consists in a finite number of geodesics (see Section \ref{s:invariance}). Note in particular that if $d_{(x,\xi)}\ml{I}_g(V)=0$, then the geodesic issued from $(x,\xi)$ is included in $\ml{C}(V)$. As an application of Theorem~\ref{t:theo-zoll} below, one has
\begin{theo}\label{t:noncon}
Suppose that $(M,g)$ is a Compact Rank-One Symmetric Space and that $\ml{C}(V)\neq M$. Then there exist infinitely many geodesics $\gamma$ of $(M,g)$ such that $\nu(\gamma)=0$ for every $\nu\in\ml{N}(\infty)$. In particular, $\delta_\gamma\not\in \mathcal{N}(\infty)$, and
$$\mathcal{N}(\infty)\neq\ml{N}_g.$$
\end{theo}
When $d=2$ and $M$ is orientable, we are able to obtain a more precise result. 
\begin{theo}\label{t:noncon2d}
Suppose that $(M,g)=(\IS^2,\operatorname{can})$. Then every $\nu\in\mathcal{N}(\infty)$ can be written as:
$$\nu= f  \operatorname{vol}_g+\nu_{\operatorname{sing}},$$
where $f\in L^1(\IS^2)$ and $\nu_{\operatorname{sing}}$ is a nonnegative measure supported on $\ml{C}(V)$. When $\ml{C}(V)$ consists in a finite number of geodesics $\gamma_1,...\gamma_n$ then one has:
$$\nu= f \operatorname{vol}_g+\sum_{j=1}^n c_j\delta_{\gamma_j},$$
for some $c_j\geq 0$.
\end{theo}

\begin{rema}
Note that the condition on  $\ml{C}(V)$ in Theorems
\ref{t:noncon} and \ref{t:noncon2d} is non-empty as soon as the Radon transforms
$\ml{I}_g(V)$ is not constant, which is the case
generically if, for instance, $(M,g)$ is the $2$-sphere endowed with its canonical metric -- see e.g. appendix~$A$ of~\cite{Gu76}.
\end{rema}

Our third result on eigenfunctions deals with Zoll surfaces that are not isometric to $(\IS^2,\operatorname{can})$. Recall from chapter~$4$ in~\cite{Be78} that the $C_{2\pi}$-metrics of revolution on $\IS^2$ are precisely those that can be written in spherical coordinates as:
$$g_{\sigma}=(1+\sigma(\cos\theta))^2d\theta^2+\sin^2\theta d\phi^2,$$
where $\sigma$ is a smooth odd function on $[-1,1]$ satisfying $\sigma(1)=0$; such surfaces are also called Tannery surfaces.
Combining our methods to some earlier results by Zelditch~\cite{Ze96, Ze97}, one can prove the following result for $C_{2\pi}$-surfaces of revolution on $\IS^2$:
\begin{theo}\label{t:revolution} Let $g_{\sigma}$ be a $C_{2\pi}$-metric on $\IS^2$ such that $\sigma'(0)\neq 0$. Then there exist infinitely many geodesics $\gamma$ of $(\IS^2,g_{\sigma})$ such that $\nu(\gamma)=0$ for every $\nu$ in $\ml{N}(\infty)$. In particular, $\delta_\gamma\not\in \mathcal{N}(\infty)$, and
$$\mathcal{N}(\infty)\neq\ml{N}_{g_{\sigma}}.$$
\end{theo}
The proof of this result will be given in paragraph~\ref{ss:revolution}. Theorems \ref{t:noncon} and \ref{t:noncon2d} will also be proved in Section \ref{s:averaging}; they are obtained as a consequence of our analysis of the time-dependent Schrödinger equation.

\subsection{Long-time dynamics for the semiclassical Schrödinger equation} Our analysis of the time-dependent Schrödinger evolution fits in a natural way in the semiclassical framework. There will be three scales involved: the characteristic scale of oscillation of the solution $\hbar$, the size of the perturbation $\eps^2$ and the time-scale $\tau$. Our goal is to describe how these three scales affect the dynamics of the Schrödinger equation.  

We shall assume that $\hbar\in(0,1]$, and define the semiclassical the Schr\"odinger operator:
\begin{equation}\label{e:perturbed-operator}
P_{\eps}(\hbar):=-\frac{\hbar^2\Delta_g}{2}+\eps_{\hbar}^2V,
\end{equation}
where, as before, $V$ belongs to $\ml{C}^{\infty}(M,\IR)$. 
We shall be interested in the regime in which the strength $\eps:=(\eps_{\hbar})_{0<\hbar\leq 1}$ of the perturbation satisfies
$$\eps_{\hbar}\rightarrow 0\ \text{as}\ \hbar\rightarrow 0^+.$$
The corresponding semiclassical Schr\"odinger equation is:
\begin{equation}\label{e:perturbed-schr}
i\hbar\partial_t v_{\hbar}=P_{\eps}(\hbar)v_{\hbar},\quad v_{\hbar}|_{t=0}=u_{\hbar}\in L^2(M).
\end{equation}
We assume that the sequence of initial data $(u_{\hbar})_{\hbar>0}$ is normalized in $L^2(M)$ and satisfies the following oscillation properties:
\begin{equation}\label{e:hosc}
\limsup_{\hbar\rightarrow 0}\left\Vert \mathbf{1}_{\left[  R,\infty\right)
}\left(  -\hbar^{2}\Delta_g\right)  u_{\hbar}\right\Vert _{L^{2}\left(  M\right)
}\longrightarrow0,\quad\text{as }R\longrightarrow\infty,
\end{equation}
and%
\begin{equation}\label{e:shosc}
\limsup_{\hbar\rightarrow 0}\left\Vert \mathbf{1}_{\left[  0,\delta\right]
}\left(  -\hbar^{2}\Delta_g\right)  u_{\hbar}\right\Vert _{L^{2}\left(  M\right)
}\longrightarrow0,\quad\text{as }\delta\longrightarrow0^{+}.
\end{equation}

\begin{rema}
Note that any normalized sequence $(u_n)_{n\in\mathbb{N}}$ safisfies (\ref{e:hosc}) for some $\left(
\hbar_{n}\right)  $; this is also the case with (\ref{e:shosc}), provided $\left(
u_{n}\right)  $ weakly converges to zero (the sequence $\left(  \hbar_{n}\right)
$ here may be different from the one in (\ref{e:hosc})). Yet, in general, it is
possible to construct normalized sequences $\left(  u_{n}\right)  $ such that
(\ref{e:hosc}), (\ref{e:shosc}) never hold simultaneously (see \cite{Ge98}). Note also that any sequence of normalized
eigenfunctions of $-\frac{1}{2}\Delta_g+V$ satisfies (\ref{e:hosc}), (\ref{e:shosc}) with $\hbar_{n}=\lambda_{n}^{-1}$ where $\lambda_n^2$ is the corresponding eigenvalue.
\end{rema}
We are interested in understanding the dynamics of the sequence of solutions $(v_\hbar)$ at time scales $\tau:=(\tau_{\hbar})_{0<\hbar\leq 1}$ where 
$$\lim_{\hbar\rightarrow 0^+}\tau_{\hbar}=+\infty.$$
More precisely, the main object in our study is the sequence of time-scaled position densities:
$$\nu_{\hbar}:(t,x)\longmapsto|v_\hbar(\tau_\hbar t,x)|^2=\left| e^{-i\frac{\tau_{\hbar}}{\hbar}t P_{\eps}(\hbar)}u_{\hbar}\right|^2(x).$$
The sequence $(\nu_\hbar)$ is bounded in $\ml{D}'(\IR\times M)$ and satisfies:
$$\left|\langle \nu_\hbar, b\rangle_{\ml{D}'\times \ml{C}^\infty}\right| \leq \int_\IR || b(t,\cdot) ||_{L^\infty(M)}dt,\ \text{ for every } b\in \ml{C}^\infty_c(\IR\times M).$$
Therefore, every accumulation point of the sequence $(\nu_\hbar)$ is an element in $L^{\infty}(\IR,\ml{P}(M))$, where $\ml{P}(M)$ denotes the set of probability measures on $M$. We shall denote by $\ml{N}(\tau,\eps)$ the set of such accumulation points obtained as the sequence of initial data $(u_{\hbar})$ varies among normalized sequences satisfying~\eqref{e:hosc} and~\eqref{e:shosc}.

Our goal is to understand the structure of the elements in $\ml{N}(\tau,\eps)$, in particular how this set depends on $\tau$ and $\eps$ as well as on the geometry of the manifold and the form of the potential. Again any element $\nu(t)$ in $\ml{N}(\tau,\eps)$ satisfies $\nu(t)\in\ml{N}_g$ for a.e. $t$ in $\IR$.
\begin{rema}
When $\eps_{\hbar}=\hbar$ and $\tau_\hbar = \hbar^{-1}$, one has 
$$e^{-i\frac{\tau_{\hbar}}{\hbar}t P_{\eps}(\hbar)} = e^{it(\frac{1}{2}\Delta_g-V)}.$$
Therefore this particular scaling corresponds to studying solutions to the non-semiclassical Schrödinger equation \eqref{e:nonsemiclassical-schr}.
\end{rema}
\begin{rema}\label{r:eigen}
If $u$ is an eigenfunction of the operator \eqref{e:eigenf} then, by the previous remark,
$$|e^{-i\frac{\tau_{\hbar}}{\hbar}t P_{\hbar}(\hbar)}u|^2=|e^{i\hbar\tau_\hbar t (\frac{1}{2}\Delta_g-V)}u|^2=|u|^2.$$
This implies that $\ml{N}(\infty) \subset \ml{N}(\tau,\hbar)$ for any time scale $\tau$.
\end{rema}

\begin{rema}
Most of our analysis can be extended to perturbations of the form $\eps_{\hbar}^2Q_{\hbar}$, where $Q_{\hbar}$ is a selfadjoint operator in $\Psi^{0,0}(M)$.
Taking $Q_{\hbar}=V$ makes the exposition slightly simpler as we can use homogeneity properties of $V$.
\end{rema}

The set $\ml{N}(\tau,\eps)$ was characterized in \cite{Ma09} in the case of short times ($\tau_{\hbar}\ll\hbar^{-2}$) and small perturbations ($\eps_\hbar \leq \hbar$) in the same geometric setting as in this article. It turns out that in that case, the elements in $\ml{N}(\tau,\eps)$ do not depend in $t$ and one has
$$\ml{N}(\tau,\eps) = \ml{N}_g,$$
that is, $\ml{N}(\tau,\eps)$ gets in some sense as big as it can be.

In the context of negatively curved surfaces, some equidistribution properties of the elements of $\ml{N}(\tau,\eps)$ were obtained for strong enough perturbations~\cite{EsRi14, Ri14}.  

This problem has also been studied in great detail for the case $M=\IT^d$ endowed with its natural metric. The articles~\cite{AFM12, AM10} describe this set in the case of small perturbations ($\eps_{\hbar}\leq\hbar$), whereas the case of stronger perturbations will be studied in more detail in~\cite{MaRi14}. In particular, in~\cite{AFM12, AM10} it is shown that when $(M,g)=(\IT^d,\operatorname{can})$ and $\eps_{\hbar}=\hbar$, the time scale $\tau_\hbar = \hbar^{-1}$ is critical for this problem. When $\tau_\hbar\ll\ \hbar^{-1}$ the set $\ml{N}(\tau,\eps)$ contains measures that are singular with respect to Lebesgue measure (in particular, it contains all $\delta_\gamma$ corresponding to closed geodesics); whereas if $\tau_\hbar \geq \hbar^{-1}$ then $\ml{N}(\tau,\eps)$ consists only of measures that are absolutely continuous with respect to the Lebesgue measure. Moreover, when $\tau_\hbar = \hbar^{-1}$ a precise description of the dependence on $t$ of those mesures is given.

In this article we show that a similar phenomenon takes place for Zoll manifolds. It turns out that the critical time scale for this problem is $\tau_\hbar = \eps_\hbar^{-2}$ in the case of strong perturbation ($\eps_{\hbar}\geq\hbar$) or $\tau_\hbar = \hbar^{-2}$ for small perturbations ($\eps_{\hbar}\ll\ \hbar$) :
\begin{itemize}
\item Below the critical time-scale, one has $\ml{N}(\tau,\eps) = \ml{N}_g,$ and we provide a formula to compute $\nu$ from the sequence of initial data.

\item At the critical time-scale, measures in  $\ml{N}(\tau,\eps)$ may depend in a non trivial way on $t$. We give an explicit propagation law that involves the Radon transform of the potential and the sequence of initial data.

\item Above the critical time-scale, the restriction of the measures in  $\ml{N}(\tau,\eps)$ to the complement of the critical set $\ml{C}(V)$ has some additional regularity. When $(M,g)=(\IS^2,\operatorname{can})$, this restriction is in fact absolutely continuous with respect to the Riemannian measure.
\end{itemize}

The precise statement of our results is given in Propositions~\ref{p:CROSS},~\ref{p:invariance} and~\ref{p:invariancesmall}; they are formulated in terms of semiclassical measures and are presented in Section \ref{ss:prop}. In Section \ref{ss:echo} we apply these results to the study of the quantum Loschmidt echo.

\subsection{Some results on the structure of $\ml{N}(\tau,\eps)$}
Let us now present some direct consequences of our results.
\begin{theo}\label{t:theo-zoll} 
Suppose $(M,g)$ is a $P_l$-manifold and that
$$\tau_{\hbar}\eps_{\hbar}^2\rightarrow +\infty.$$
Suppose one the following condition holds:
\begin{enumerate}
 \item $\eps_{\hbar}\hbar^{-1}\rightarrow +\infty$;
 \item $(M,g)$ is a compact rank one symmetric space.
\end{enumerate}
Then, for every $\nu$ in ${\ml{N}}(\tau,\eps)$ and for every geodesic $\gamma_0$ that is not contained in
$\ml{C}(V)$, one has:
$$\nu(t)(\gamma_0)=0,\ \text{for a.e.}\ t\in\IR.$$
\end{theo}

We note that for general $P_l$ manifolds we are still able to obtain a similar result even if the condition $\eps_{\hbar}\gg\hbar$ does not hold -- see section~\ref{s:main} for more details. It is however more complicated to state as it involves the structure of a remainder term appearing in the natural decomposition of $\sqrt{-\Delta_g}$ on Zoll manifolds~\cite{CdV79,DuGu75}. In the case of a general $P_l$-metric on $M=\mathbb{S}^2$, an explicit expression of this term involving curvature terms and Jacobi fields was obtained by Zelditch in~\cite{Ze96, Ze97} (see Paragraph~\ref{ss:revolution} for more details on this issue).

This Theorem is a direct consequence of Propositions~\ref{p:CROSS} and~\ref{p:invariance} and it will be proved in Paragraph~\ref{ss:proof-theo-zoll}. It tells us in particular that, for large enough times, solutions of the Schr\"odinger equation cannot be concentrated on closed geodesics corresponding to regular points of $\ml{I}_g(V)$. For instance, if one considers a sequence $(u_{\hbar})_{0<\hbar\leq 1}$ of coherent states that is microlocalized at a certain point $\rho\in\mathring{T}^*M$, then the corresponding solution will not be concentrated along the corresponding closed geodesic (for large enough times) provided that the geodesic consists of regular points. 

Theorem \ref{t:theo-zoll} admits the following reformulation in terms of quasimodes for the Schr\"odinger operator $P_{\eps}(\hbar)$:
\begin{coro}\label{c:zoll-coro}
Suppose the hypotheses of Theorem \ref{t:theo-zoll} hold.
Let $(\psi_{\hbar})_{0<\hbar\leq 1}$ be a normalized sequence in $L^2(M)$ satisfying
$$\left(-\frac{\hbar^2\Delta_g}{2}+\eps_{\hbar}^2 V\right)\psi_{\hbar}=E(\hbar)\psi_{\hbar}+o(\hbar\eps_{\hbar}^2),$$
with $E(\hbar)\rightarrow E\neq 0$ as $\hbar\rightarrow 0.$ 
Then, for every weak-$\star$ accumulation point $\nu_0$ of the sequence
$$\nu_{\hbar}:=|\psi_{\hbar}|^2 \operatorname{vol}_g,$$
and for every geodesic $\gamma_0$ that is not contained in $\ml{C}(V)$, one has
$$\nu_0(\gamma_0)=0.$$ 
\end{coro}
We would like to stress the fact that we are not requiring $\eps_{\hbar}\ll\sqrt{\hbar}$ for our results to hold. One often makes this assumption when studying the spectral properties of semiclassical operators with periodic bicharacteritics (see~\cite{OV10, HaHiSj14} for instance) in order to to keep the nice cluster structure of the the spectrum of the operator $-\Delta_g$, see~\cite{CdV79}. 

For the $2$-sphere endowed with its canonical metric $g=\text{can}$,
we are able to be more precise regarding the regularity of the elements
in $\ml{N}(\tau,\eps)$:

\begin{theo}\label{t:theo-S2} Let $V$ be a smooth function on $(\IS^2,\operatorname{can})$ such that $\ml{C}(V)$ consists in a finite number of geodesics
$\gamma_1,\ldots,\gamma_r$. Suppose
$$\tau_{\hbar}\eps_{\hbar}^2\rightarrow +\infty.$$
Then, any $\nu \in \ml{N}(\tau,\eps)$ is of the following form
$$\nu(t,\cdot)=f(t,\cdot) \operatorname{vol}_{\operatorname{can}}+\sum_{j=1}^r c_j(t)\delta_{\gamma_j},$$
where for a.e. $t$ in $\IR$, $f(t,\cdot) \in L^1(\IS^2)$ and $c_1,\ldots, c_r$ are non-negative functions in $L^\infty(\IR)$.
\end{theo}

\begin{rema}
This Theorem is a direct consequence of
Proposition~\ref{p:invariance} and Corollary~\ref{c:struct}.
\end{rema}

\begin{rema} Recall that the space of geodesics on $\IS^2$ can be identified with~$\IS^2$~\cite{Be78} (remark $2.10$) and that $\ml{I}_{\text{can}}$ induces an isomorphism
from $\ml{C}^{\infty}_{\text{even}}(\IS^2)$ onto $\ml{C}^{\infty}_{\text{even}}(\IS^2)$~\cite{Gu76}. Moreover, $\ml{I}_{\text{can}}(V)=0$ for any odd function on $\IS^2$.
In particular, for a generic choice of $V$, the assumption of the Theorem is satisfied.
\end{rema}

\subsection{An application to spectral theory}

We mention the following Proposition on the spectral properties of $P_{\eps}(\hbar)$ which can also be obtained using the tools developed in the present article:

\begin{prop}\label{p:level-spacing}  Suppose $(M,g)$ is a $P_l$-manifold. Let $E_0>0$ and let $0<2\delta_0\leq E_0$. Let $(E_j(\hbar))_{j=1}^{N(\hbar)}$ be the distinct eigenvalues of $P_{\eps}(\hbar)$ in the interval $[E_0-\delta_0,E_0+\delta_0]$.

If $\ml{C}(V)\neq M$, then there exists some constant $c_0>0$ such that, for $\hbar>0$ small enough, one has
 $$s_0(\hbar):=\inf\{|E_j(\hbar)-E_k(\hbar)|:1\leq j\neq k \leq N(\hbar)\}\leq c_0\hbar\eps_{\hbar}^2,$$
 provided one of the following conditions holds:
 \begin{enumerate}
  \item $\eps_{\hbar}\hbar^{-1}\rightarrow +\infty$;
  \item $(M,g)$ is a compact rank-one symmetric space.
 \end{enumerate}
\end{prop}

The proof of this result will be given in Section~\ref{s:level-spacing}. We note that, thanks to the semiclassical Weyl's law~\cite{DiSj99}, one knows that $s_0(\hbar)>0$. In the case where $d=2$ and $\hbar^N\ll\eps_{\hbar}\ll\sqrt{\hbar}$ (where $N$ is some positive exponent related to the clustering of the unperturbed operator), a much stronger result was for instance obtained in~\cite{HaHiSj14}. In fact, it was proved in this reference that, near regular values of $\ml{I}_g(V)$, one can obtain an asymptotic expansion of the eigenvalues with a level spacing which is exactly of order $\hbar\eps_{\hbar}^2$. Compared with this result, the above Proposition only provides a simple criterium for the existence of distinct eigenvalues which are asymptotically at distance less than $\hbar\eps_{\hbar}^2$. On the other hand, it is valid in any dimension and even for strong perturbation, meaning $\eps_{\hbar}\geq\sqrt{\hbar}$. 

\bigskip
\noindent\textbf{Acknowledgments.} Much of this work was done while the first author was visiting Laboratoire Paul Painlev\'e at Université de Lille 1 in October 2013 in the framework of the CEMPI program (ANR-11-LABX-0007-01). He wishes to thank this institution for its warm hospitality and support. The authors thank San V\~u Ng\d{o}c for useful discussions on integrable systems.
\section{Main results: propagation of semiclassical measures}\label{s:main}

We now describe our main results, that are formulated using the notion of semiclassical or Wigner measure.

\subsection{Semiclassical measures} In order to study the asymptotic properties of the solutions of~\eqref{e:perturbed-schr}, we will make use of the so-called semiclassical measures~\cite{Ge91}, or more precisely of their time dependent version~\cite{Ma09} -- see also appendix~\ref{a:measures} for a brief reminder. We now recall their construction. For a given $t$ in $\IR$, we denote the Wigner distribution at time $t\tau_{\hbar}$ by
\begin{equation}\label{e:wigner}w_{\hbar}(t\tau_{\hbar})(a):=\left\la v_{\hbar}(t\tau_{\hbar}),\Oph(a)v_{\hbar}(t\tau_{\hbar})\right\ra,
\end{equation}
where $\Oph(a)$ is a $\hbar$-pseudodifferential operator with principal symbol $a\in \ml{C}^{\infty}_c(T^*M)$ -- see appendix~\ref{a:pdo}. This quantity represents the distribution of the solution of~\eqref{e:perturbed-schr} in the phase space $T^*M$. 

Recall now that we can extract a subsequence $\hbar_n\rightarrow 0^+$ as $n\rightarrow +\infty$ such that, for every $a$ in $\ml{C}^{\infty}_c( T^*M)$ and for every $\theta$ in $L^1(\IR)$,
$$\lim_{\hbar_n\rightarrow 0^+}\int_{\IR\times T^*M}\theta(t)a(x,\xi) w_{\hbar_n}(t\tau_{\hbar_n},dx,d\xi)dt=\int_{\IR\times T^*M}\theta(t)a(x,\xi)\mu(t,dx,d\xi)dt,$$
where $(t,x,\xi)\mapsto\mu(t,x,\xi)$ belongs to $L^{\infty}(\IR,\ml{M}( T^*M))$, with $\ml{M}( T^*M)$ the set of finite complex measures carried on $T^*M$. Recall also that, for a.e. $t\in\IR$, $\mu(t,\cdot)$ is in fact a \emph{probability measure} which is carried on $\mathring{T}^*M$ and which is \emph{invariant by the geodesic flow} $\varphi^s$ on $T^*M$. For instance, $\mu(t,\cdot)$ can be the normalized Lebesgue measure along closed orbit of the geodesic flow. We refer to appendix~\ref{a:measures} for a brief reminder of these results from~\cite{Ma09}.

We will denote by $\ml{M}(\tau,\eps)$ the set of accumulation
points of the sequences 
$$\mu_{\hbar}:(t,x,\xi)\longmapsto
w_{\hbar}(t\tau_{\hbar},x,\xi)$$ as $(u_{\hbar})$ 
varies
among normalized sequences satisfying~\eqref{e:hosc}
and~\eqref{e:shosc}. 

\begin{rema}  
Thanks to the frequency assumption~\eqref{e:hosc}, one can also
verify that $\ml{N}(\tau,\eps)$ corresponds in fact to the
projections on $M$ of the elements of $\ml{M}(\tau,\eps)$.
\end{rema}

\subsection{Propagation at different time scales}\label{ss:prop}

The Zoll structure on $M$
allows to prove that every element in $\ml{M}(\tau,\eps)$ satisfies
an additional invariance or transport property, depending on the
relative size of $\tau$ and $\eps$. 

We denote by
$\varphi^t_V$ the Hamiltonian flow associated to the Radon transform of the potential $\ml{I}_g(V)$.
Note that $\varphi^{t}_V$ commutes with the geodesic flow
$\varphi^s$.

Let us start presenting our results in the particular case when
$(M,g)$ is a Compact Rank-One Symmetric Space, since they are
somewhat simpler to describe. It turns out, as it was already
stated in the introduction, that the time scale
$\tau_{\hbar}=\eps_{\hbar}^{-2}$ is critical for this problem.
More precisely, the following result holds:

\begin{prop}\label{p:CROSS}
Suppose $(M,g)$ is a Compact Rank-One Symmetric Space. Let $\mu
\in \ml{M}(\tau,\eps)$ and denote by $\mu_0$ the semiclassical
measure of the sequence of initial data used to generate $\mu$.
The following results hold.

\medskip
\noindent i) If $\tau_{\hbar}\eps_{\hbar}^2\rightarrow 0^+$ then
$\mu$ is continuous with respect to $t$ and, for every $b\in
\ml{C}^{\infty}_c(T^*M)$ and every $t\in\IR $:
\begin{equation*}
\mu(t)(b)=\mu_0(\ml{I}_g(b)).
\end{equation*}

\medskip
\noindent ii) If $\tau_{\hbar}\eps_{\hbar}^2=1$ then $\mu$ is
continuous with respect to $t$ and, for every $b\in
\ml{C}^{\infty}_c(T^*M)$ and every $t\in\IR $:
\begin{equation*}
\mu(t)(b)=\mu_0(\ml{I}_g(b)\circ\varphi_{V}^t).
\end{equation*}

\medskip
\noindent iii) If $\tau_{\hbar}\eps_{\hbar}^2\rightarrow +\infty$
then $\mu$ has an additional invariance property. For almost every
$t\in\IR$ and every $s\in\IR$:
\begin{equation*}
(\varphi_{V}^s)_*\mu(t)=\mu(t).
\end{equation*}
\end{prop}

This result can be obtained by similar arguments as the proof of Propositions
\ref{p:invariance}, \ref{p:invariancesmall} presented below, which
describe additional properties satisfied by the elements of
$\ml{M}(\tau,\eps)$ in the more general case of $P_l$-manifolds. In this case, one must take into account a certain $(\varphi^s)$-invariant function $q_0$ on $\mathring{T}^*M$ that depends on the metric $g$. Let
$\varphi_{q_0}^t$ and $\varphi_{V+q_0}^t$ denote respectively the
Hamiltonian flows on $\mathring{T}^*M$ generated by $q_0$ and
$\ml{I}_g(V)+q_0$.

The first of these result is concerned with the ``big perturbation"
regime, when the square root of the size of the perturbation
dominates the characteristic length scale of oscillations. 

\begin{prop}\label{p:invariance}
Suppose $M$ is endowed with a $P_l$-metric $g$ and that:
$$\eps_{\hbar}\ge\hbar.$$
There exists a smooth, $0$-homogeneous and $(\varphi^s)$-invariant
function $q_0$ defined on $\mathring{T}^*M$ (and depending only on $(M,g)$)
such that, for every $\mu\in\ml{M}(\tau,\eps)$ associated to a
sequence of initial data with an unique semiclassical measure
$\mu_0$, the following holds:

\medskip
\noindent i) If $\tau_{\hbar}\eps_{\hbar}^2\rightarrow 0^+$, then,
$\mu$ is continuous in the $t$ variable and, for every $b\in
\ml{C}^{\infty}_c(T^*M)$ and every $t\in\IR $:
\begin{equation} \label{e:smalle}
\mu(t)(b)=\mu_0(\ml{I}_g(b)).
\end{equation}

\medskip
\noindent ii) If $\tau_{\hbar}\eps_{\hbar}^2=1$ then $\mu$ is
continuous in the $t$ variable and, for every $b\in
\ml{C}^{\infty}_c(T^*M)$ and every $t\in\IR $:
\begin{equation}\label{e:criticale1}
\mu(t)(b)=\mu_0(\ml{I}_g(b)\circ\varphi_{V+q_0}^t), \text{ if }
\eps_{\hbar} = \hbar,
\end{equation}
\begin{equation}\label{e:criticale2}
\mu(t)(b)=\mu_0(\ml{I}_g(b)\circ\varphi_{V}^t), \text{ if }
\eps_{\hbar}\hbar^{-1}\rightarrow +\infty.
\end{equation}
\noindent iii) If $\tau_{\hbar}\eps_{\hbar}^2\rightarrow +\infty$,
then $\mu$ has an additional invariance property. For almost every
$t\in\IR$ and every $s\in\IR$:
\begin{equation}\label{e:bige1}
(\varphi_{V+q_0}^s)_*\mu(t)=\mu(t), \text{ if } \eps_{\hbar} = \hbar,
\end{equation}
\begin{equation}\label{e:bige2}
(\varphi_{V}^s)_*\mu(t)=\mu(t), \text{ if }
\eps_{\hbar}\hbar^{-1}\rightarrow +\infty.
\end{equation}
\end{prop}

The next result deals with the ``small perturbation" regime, when
$\hbar$ dominates $\eps_\hbar$ (and therefore, the time scale must
be compared to $\hbar^2$ instead of $\eps_{\hbar}^2$). As it can
be expected, the effect of the perturbation in this case is
negligible. We however present the precise statement for the sake
of completeness.

\begin{prop}\label{p:invariancesmall}
Suppose $M$ is endowed with a $P_l$-metric $g$ and let $q_0$ be
the function introduced in Proposition \ref{p:invariance}. Suppose
that: $$\eps_{\hbar}\hbar^{-1}\rightarrow 0^+.$$ Then, for every
$\mu\in\ml{M}(\tau,\eps)$ associated to a sequence of initial data
with an unique semiclassical measure $\mu_0$, the following holds:

\medskip
\noindent i) If $\tau_{\hbar}\hbar^2\rightarrow 0^+$, then, $\mu$
is continuous in the $t$ variable and, for every $b\in
\ml{C}^{\infty}_c(T^*M)$ and every $t\in\IR $:
\begin{equation} \label{e:smalle}
\mu(t)(b)=\mu_0(\ml{I}_g(b)).
\end{equation}

\medskip
\noindent ii) If $\tau_{\hbar}\hbar^2=1$ then $\mu$ is continuous
in the $t$ variable and, for every $b\in \ml{C}^{\infty}_c(T^*M)$
and every $t\in\IR $:
\begin{equation}\label{e:criticale0}
\mu(t)(b)=\mu_0(\ml{I}_g(b)\circ\varphi_{q_0}^t).
\end{equation}

\medskip
\noindent iii) If $\tau_{\hbar}\hbar^2\rightarrow +\infty$, then
$\mu$ has an additional invariance property. For almost every
$t\in\IR$ and every $s\in\IR$:
\begin{equation}\label{e:bige0}
(\varphi_{q_0}^s)_*\mu(t)=\mu(t).
\end{equation}
\end{prop}

We note that the statements of these Propositions remain valid for
$V\equiv 0$ (in fact, the results of Proposition
\ref{p:invariance} are contained in those of Proposition
\ref{p:invariancesmall} in this case).

As it will become clear from our proof, the symbol $q_0$ is
related to the ``natural'' decomposition of the operator
$\sqrt{-\Delta_g}$ on Zoll manifolds -- see for instance Theorem
$1.1$ of~\cite{CdV79}. As stated in the Proposition, this symbol
depends only on the choice of the metric. We emphasize yet that it
is given by an hardly explicit formula -- see
Remark~\ref{r:zelditch} or~\cite{Ze96, Ze97}.

However, when $(M,g)$ is a Compact Rank-One Symmetric Space, one
can take $q_0=0$ and therefore it is possible to derive the result in
Proposition \ref{p:CROSS} without making any assumption regarding
the relative size of $\hbar$ and $\eps_\hbar$.

\begin{rema}
When $\ml{I}_g(V)$ is constant, the Hamiltonian flow $\varphi^t_V$ acts trivially, and some of the above statements are thus empty. This is the case, for instance, when $(M,g)=(\IS^d,\operatorname{can})$
and $V=V(x)$ does not depend on $\xi$ and is an odd function plus
a constant (see \cite{HelgBook}, Theorems 1.17, 1.23).
\end{rema}

\subsection{Application to the study of the quantum Loschmidt echo}\label{ss:echo} 
As another application of the methods developped to prove Proposition~\ref{p:CROSS},~\ref{p:invariance} and~\ref{p:invariancesmall}, we derive some properties on the so-called quantum Loschmidt echo. This quantity is defined as follows
$$\ml{E}_{\hbar}(t\tau_{\hbar}):=|\langle v_{\hbar}(t\tau_{\hbar}),v_{\hbar}^0(t\tau_{\hbar})\rangle|^2,$$
where 
\begin{itemize}
\item $v_{\hbar}(t\tau_{\hbar})$ is the solution of~\eqref{e:perturbed-schr};
\item $v_{\hbar}^0(t\tau_{\hbar})$ is the solution of~\eqref{e:perturbed-schr} when we pick $V\equiv 0$.
\end{itemize}
This quantity was introduced by Peres in~\cite{Pe84} and it allows to measure the sensitivity of a quantum system to perturbations of the Hamiltonian. Peres predicted that this quantity should typically goes to $0$ and that the decay rate indicates the chaotic or integrable nature of the underlying classical system. Since this seminal work of Peres~\cite{Pe84}, many progresses have been made in the physics literature regarding the asymptotic properties of this quantity, especially in the context of chaotic systems. We refer the reader to~\cite{GPSZ06, JaPe09, GJPW12} for recent surveys on these questions. Our approach allows to study a slightly related quantity:
$$F_{\hbar}(t\tau_{\hbar}):=\langle v_{\hbar}(t\tau_{\hbar}),v_{\hbar}^0(t\tau_{\hbar})\rangle.$$
Up to an extraction $\hbar_n\rightarrow 0$, one can suppose that there exists $F(t)$ in $\ml{D}'(\IR)$ such that, for every $\theta$ in $\ml{C}^{\infty}_c(\IR)$,
$$\lim_{\hbar_n\rightarrow 0}\int_{\IR}\theta(t)F_{\hbar_n}(t\tau_{\hbar_n})dt=\int_{\IR}\theta(t)F(t)dt.$$
Our last result gives a description of the limit distribution $F(t)$ in the context of $P_l$-manifolds:
\begin{prop}\label{t:loschmidt} Suppose $(M,g)$ is a $P_l$-manifold and that
$$\lim_{\hbar\rightarrow 0^+}\eps_{\hbar}\hbar^{-\frac{1}{2}}=0.$$
Suppose also that one the following condition holds:
\begin{enumerate}
 \item $\eps_{\hbar}\hbar^{-\frac{3}{2}}\rightarrow +\infty$,
 \item $(M,g)$ is a compact rank one symmetric space.
\end{enumerate}
Then, for every $t\mapsto F(t)$ associated to a sequence of initial data with an unique semiclassical measure $\mu_0$, the following holds

\medskip
\noindent i) If $\lim_{\hbar\rightarrow 0^+}\frac{\tau_{\hbar}\eps_{\hbar}^2}{\hbar}=0$, then, for every $t$ in $\IR$:
$$F(t)=1.$$
\medskip
\noindent ii)  If $\tau_{\hbar}=\frac{\hbar}{\eps_{\hbar}^2}$, then, $F(t)$ is continuous in the $t$ variable and, for every $t$ in $\IR$:
$$F(t)=\mu_0\left(e^{it\ml{I}_g(V)}\right).$$
\medskip
\noindent iii) If $\lim_{\hbar\rightarrow 0^+}\frac{\tau_{\hbar}\eps_{\hbar}^2}{\hbar}=+\infty$ and if $\{\ml{I}_g(V)=0\}=\emptyset$, then, for every $t$ in $\IR$,
$$F(t)=0.$$
\end{prop}

This Theorem will follow from Proposition~\ref{p:loschmidt} and its proof will make use of similar tools as the ones used to prove Proposition~\ref{p:CROSS},~\ref{p:invariance} and~\ref{p:invariancesmall}. Note that other mathematical studies of the long time properties of the quantum Loschmidt echo appeared recently in various geometric settings: $1$-dimensional systems~\cite{BolSc06, CoRo07}, negatively curved surfaces~\cite{Ri14}.

\section{Averaging, transport and invariance}

\label{s:averaging}

In this section, we will prove Propositions~\ref{p:CROSS},~\ref{p:invariance} and~\ref{p:invariancesmall}. It is organized as follows. First, we recall an averaging procedure due to Weinstein~\cite{We77} which we formulate in a semiclassical language following~\cite{Ch80, HeRo84, DiSj99}. Then, in paragraph~\ref{ss:invariance}, we deduce the above Propositions from this averaging Lemma, and, in paragraph~\ref{ss:proof-theo-zoll}, we derive the proof of a slightly stronger version of Theorem~\ref{t:theo-zoll} from these Propositions.

\subsection{Semiclassical averaging Lemma}

The following result is a quantum analogue of the averaging method
for classical dynamical systems.

\begin{theo}\label{t:av}
Let $(M,g)$ be a $P_l$-manifold. Then, for every $b \in
\ml{C}^{\infty}_c(\mathring{T}^*M)$, there exists an operator $\la B_{\hbar}\ra \in \Psi^{0,0}(M)$
whose principal symbol is the classical average:
\begin{equation}\label{e:ca}
\ml{I}_g(b)(x,\xi):=\frac{\|\xi\|_x}{l}\int_0^{\frac{l}{\|\xi\|_x}}b\circ
\varphi^{s}(x,\xi)ds,
\end{equation}
and which satisfies:
\begin{equation}\label{e:cm}
[\la B_{\hbar}\ra, P_0(\hbar)] = \ml{O}_{L^2 \rightarrow
L^2}(\hbar^{3}).
\end{equation}
In addition, if $(M,g)$ is a Compact Rank-One Symmetric Space,
$\la B_{\hbar}\ra$ can be chosen such that the above formula is
exact, that is:
\begin{equation}\label{e:cmCROSS}
[\la B_{\hbar}\ra, P_0(\hbar)] = 0.
\end{equation}
\end{theo}

This type of result is rather well-known and goes back
to~\cite{We77, DuGu75, CdV79}. We will give here a semiclassical
version of the argument presented in those references following
the presentation of~\cite{DiSj99} (chapter $15$) -- see
also~\cite{Ch80, HeRo84} for a semiclassical treatment.

The proof of Theorem \ref{t:av} will be done in three steps.

\subsubsection{Reparametrization of the classical Hamiltonian}

We have made the assumption that all the geodesics of $M$ have a common least period $l>0$. This means that, for every
$(x,\xi)$ in $\mathring{T}^*M$, one has
$\varphi^{\frac{l}{\sqrt{2E}}}(x,\xi)=(x,\xi),$ where
$E=\frac{\|\xi\|_x^2}{2}$. Following~\cite{DiSj99}, the first step
will be to ``reparametrize'' the Hamiltonian, both at the quantum
and at the classical level, in order to have a common period
$2\pi$ for the flow on $\mathring{T}^*M$. For that purpose, we set
$$\tilde{P}_0(\hbar):=\frac{l}{2\pi}\sqrt{-\hbar^2\Delta_g}.$$
The Hamiltonian corresponding to this operator is given by $\tilde{p}_0(x,\xi):=\|\xi\|_x$, and the
Hamiltonian vector fields of $p_0$ and $\tilde{p}_0$ are related
in the following way:
$$X_{\tilde{p}_0}(x,\xi)=\frac{l}{2\pi\sqrt{2E}}X_{p_0}(x,\xi),$$
where $E=\frac{\|\xi\|_x^2}{2}$. If we denote by $\psi^t$ the
Hamiltonian flow of $\tilde{p}_0$, we find that, for every
$(x,\xi)$ in $\mathring{T}^*M$ and for every $s$ in $\IR$,
one has
$$\psi^{s}(x,\xi)=\varphi^{\frac{l}{2\pi\|\xi\|_x}s}\left(
x,\xi \right), \text{ and } \psi^{2\pi}(x,\xi)=(x,\xi).$$

\subsubsection{Periodicity of the quantum propagator}

We recall that the Fourier integral operator associated to the
Hamiltonian vector field satisfies a certain periodicity property.
This follows from the periodicity of the classical flow. We will
denote by $\alpha\in\IZ$ the common Maslov index of the closed
trajectories of $\psi^t$ on the energy layers $\tilde{p}_0^{-1}\left((0,+\infty)\right).$ According to Lemma~$29.2.1$ in~\cite{Ho85IV} -- see also~\cite{CdV79} or~\cite{HeRo84}, one knows that there exists a polyhomogeneous pseudodifferential operator $A$ of order $1$ and a polyhomogeneous pseudodifferential operator $Q$ of order $-1$ (see Ch.~$18$ in~\cite{Ho85III} for the precise definitions) such that the following holds:
\begin{enumerate}
 \item $\frac{l}{2\pi}\sqrt{-\Delta_g}=A+\frac{\alpha}{4}+Q$;
 \item $[\sqrt{-\Delta_g},Q]=0$;
 \item $\text{Sp}(A)\subset\IN$.
\end{enumerate}
Translated in our semiclassical framework, we get
$$\tilde{P}_0(\hbar)=\hbar A+\frac{\alpha \hbar}{4}+\hbar Q,$$
and, in particular,
$$e^{-\frac{2i\pi}{\hbar}\left(\tilde{P}_0(\hbar)-\frac{\hbar\alpha}{4}-\hbar Q\right)}=\text{Id}_{L^2(M)}.$$

\begin{rema}\label{r:CROSS} In the case of a compact rank one symmetric space, the situation is slightly simpler using the explicit description of the
spectrum -- see paragraph $8.8$ in~\cite{Be78}. In fact, we can
write that
$$-\tilde{\Delta}_g:=-\Delta_g+\left(\frac{\pi\alpha}{2l}\right)^2\text{Id}=\left(\frac{2\pi}{l}\right)^2\left(\tilde{A}+\frac{\alpha}{4}\right)^2,$$
where $\tilde{A}$ is a polyhomogeneous pseudodifferential operator of order
one such that $\text{Sp}(\tilde{A})\subset \IN$, and then
\begin{equation}\label{e:period-cross}
e^{- \frac{i2
\pi}{\hbar}\left(\frac{l}{2\pi}\sqrt{-\hbar^2\tilde{\Delta}_g}-\frac{\hbar\alpha}{4}\right)}=e^{-i2\pi
\tilde{A}}=\text{Id}_{L^2(M)}.
\end{equation}
\end{rema}

\subsubsection{Averaging procedure}\label{sss:commute}

We will now verify that the operator $\tilde{P}_0(\hbar)$
satisfies the commutation property~\eqref{e:cm} where the operator
$\la B_{\hbar}\ra$ is defined as the quantum average of the
pseudodifferential operator $\Oph(b)$ by the Fourier integral
operator associated to $\hbar A$. This kind of
averaging procedure is standard in this context and it seems that
it first appeared in Weinstein's article~\cite{We77}.

Let $\delta<\frac{1}{2}<R$. For every symbol $b\in\ml{C}^{\infty}_c(T^{*}M)$ which
is supported in $p_0^{-1}([\delta,R])$, we define the following
averaged operator:
$$\la B_{\hbar}\ra:=\frac{1}{2\pi}\int_0^{2\pi}e^{itA}\Oph(b)e^{-itA}dt.$$
According to Egorov's Theorem, $\la B_{\hbar}\ra$ is an element in
$\Psi^{0,0}(M)$ whose principal symbol is equal to $\ml{I}_g(b)$,
which is a smooth function since $b$ is supported in
$p_0^{-1}([\delta,R])$ and therefore:
$$\ml{I}_g(b)(x,\xi) = \frac{1}{2\pi}\int_0^{2\pi}b\circ
\psi^{s}(x,\xi)ds.$$ Using the fact that $\text{Sp}(A)\subset\IN$, we start noticing that the following
commutation relation holds:
\begin{equation}\label{e:commutation}
[\la B_{\hbar}\ra,A]=0.
\end{equation} 
Let $0\leq \chi\leq 1$ be a smooth cut-off function in
$\displaystyle\ml{C}^{\infty}_c\left(\left[\frac{l}{2\pi}\sqrt{\delta},
\frac{l}{2\pi}\sqrt{4R}\right]\right)$ satisfying
$\chi\equiv 1$ in a small neighborhood of $\left[\frac{l}{2\pi}\sqrt{2\delta},
\frac{l}{2\pi}\sqrt{2R}\right]$. We can now use property~\eqref{e:commutation} in order to derive
an expression for the commutator $[\la
B_{\hbar}\ra,P_{0}(\hbar)]$. First note that:
$$[\la B_{\hbar}\ra,P_{0}(\hbar)] = [\la B_{\hbar}\ra,\chi(\tilde{P}_0(\hbar))P_{0}(\hbar)]+\ml{O}(\hbar^{\infty}).$$
Recall now that $P_0(\hbar)=\frac{(2\pi)^2}{l^2}\tilde{P}_0(\hbar)^2.$ We use the definition of $A$ in order to write:
$$\tilde{P}_0^2(\hbar)=\hbar^2A^2+\frac{\hbar^2\alpha^2}{16}+\frac{\hbar^2\alpha}{2}A+\hbar^2R_0,$$
where
\begin{equation}\label{e:qh}
Q_0:=(AQ+QA)+\frac{\alpha}{2}Q+Q^2,
\end{equation}
which also defines an operator commuting with $\tilde{P}_0(\hbar)$. Therefore, taking into account identity~\eqref{e:commutation}, we
finally obtain:
\begin{equation}\label{e:comm-bis}
[\la B_{\hbar}\ra,P_{0}(\hbar)]=\frac{2\pi^2\hbar^2}{l^2}[\la
B_{\hbar}\ra,\chi(\tilde{P}_0(\hbar))Q_0]+\ml{O}(\hbar^{\infty}).
\end{equation}
Using pseudodifferential calculus rules, one knows that $\frac{2\pi^2\hbar^2}{l^2}\chi(\tilde{P}_0(\hbar))Q_0$ is an element in $\Psi^{-\infty,0}(M)$ with principal symbol equal to $\chi(\|\xi\| l/(2\pi)) q_0(x,\xi)$ where $q_0$ is the principal symbol of the polyhomogeneous pseudodifferential operator $(AQ+QA)(2\pi)^2/l^2$ of degree $0$. Recall that it defines a smooth function on $\mathring{T}^*M$ which is $0$-homogeneous. Note that this function depends only on $(M,g)$. 
Thus, using the composition formula for
pseudodifferential operators and the Calder\'on-Vaillancourt
Theorem~\cite{Zw12}, we conclude that:
$$[\la B_{\hbar}\ra,P_{0}(\hbar)]=\ml{O}_{L^2\rightarrow L^2}(\hbar^{3}),$$
as we wanted to prove.

\begin{rema}\label{r:CROSS-bis} In the case that $(M,g)$ is a Compact Rank-One Symmetric Space, we can use Remark~\ref{r:CROSS} and mimick the above proof in order to obtain an exact commutation formula:
$$[\la B_{\hbar}\ra,P_{0}(\hbar)]=0,$$
where
 $$\la B_{\hbar}\ra:=\frac{1}{2\pi}\int_0^{2\pi}e^{ \frac{ i t}{\hbar}\left(\frac{l}{2\pi}\sqrt{-\hbar^2\tilde{\Delta}_g}-\frac{\hbar\alpha}{4}\right)}\Oph(b)
e^{- \frac{ it}{\hbar}\left(\frac{l}{2\pi}\sqrt{-\hbar^2\tilde{\Delta}_g}-\frac{\hbar\alpha}{4}\right)}dt.$$
\end{rema}

\subsection{Invariance and transport: proof of Propositions~\ref{p:CROSS},~\ref{p:invariance}, and~\ref{p:invariancesmall}}\label{ss:invariance}
We will now use Theorem \ref{t:av} to prove additional invariance
properties satisfied by every element $\mu$ in
$\ml{M}(\tau,\eps)$. Recall that, for almost every $t$ in $\IR$,
$\mu(t,\cdot)$ is a probability measure which is invariant by the
geodesic flow $\varphi^s$. We also emphasize that, for almost
every $t$ in $\IR$, one has $\mu(t)(\mathring{T}^*M)=1$ thanks to
~\eqref{e:hosc} and~\eqref{e:shosc}. These conditions imply that
for every $b\in \ml{C}^{\infty}_c(\mathring{T}^*M)$ and almost every $t$ one
has:
\begin{equation}\label{e:inv}
\mu(t)(b)=\mu(t)(\ml{I}_g(b)).
\end{equation}

Let
$\varphi_{V}^t$, $\varphi_{q_0}^t$ and $\varphi_{V+q_0}^t$ denote respectively the
Hamiltonian flows on $\mathring{T}^*M$ generated by $\ml{I}_g(V)$, $q_0$ and
$\ml{I}_g(V)+q_0$. In the case of a Compact Rank One Symmetric Space, we use the convention $q_0\equiv 0$. Note that $\varphi_{V}^t$, $\varphi_{q_0}^t$ and
$\varphi_{V+q_0}^t$ commute with the geodesic flow, since
$\{p_0,\ml{I}_g(V)\}=\{p_0,q_0\}=0 $.

We will now give the proof of Propositions~\ref{p:CROSS},~\ref{p:invariance}, and~\ref{p:invariancesmall}.
Let $\mu \in \ml{M}(\tau,\eps)$. This means that there exists a
sequence $(u_{\hbar})_{0<\hbar\leq 1}$ of normalized states in
$L^2(M)$ satisfying the frequency assumptions~\eqref{e:hosc}
and~\eqref{e:shosc} such that the Wigner distributions
$\mu_\hbar:t\mapsto w_{\hbar}(t\tau_{\hbar})$ corresponding to the
solutions $e^{-\frac{it\tau_{\hbar}P_{\eps}(\hbar)}{\hbar}}
u_{\hbar}$ converge to $\mu$ -- see appendix~\ref{a:measures}.

Assume moreover that $(u_{\hbar})_{0<\hbar\leq 1}$ has a
semiclassical measure $\mu_0$, meaning that the Wigner
distributions $w_{\hbar}(0)$ weak-$\star$ converge towards $\mu_0$
as $\hbar\rightarrow 0^+$. Let $b$ be a smooth compactly supported function on
$\mathring{T}^{*}M$. 
\begin{rema}
 We note that, as $\mu(t)(M\times\{0\})=0$ a.e., we can restrict ourselves to proving the invariance properties for such functions $b$.
\end{rema}
Let $\la
B_{\hbar}\ra$ be the operator obtained from $b$ using Theorem
\ref{t:av}. One clearly has:
\begin{equation}\label{e:wigeq}
\frac{d}{dt}\left\la
u_{\hbar},e^{\frac{it\tau_{\hbar}}{\hbar}P_{\eps}(\hbar)} \la
B_{\hbar}\ra e^{-\frac{it\tau_{\hbar}}{\hbar}P_{\eps}(\hbar)}
u_{\hbar}\right\ra = \frac{i\tau_{\hbar}}{\hbar}\left\la
u_{\hbar},e^{\frac{it\tau_{\hbar}}{\hbar}P_{\eps}(\hbar)}
[P_{\eps}(\hbar),\la B_{\hbar}\ra]
e^{-\frac{it\tau_{\hbar}}{\hbar}P_{\eps}(\hbar)}
u_{\hbar}\right\ra.
\end{equation}
Recall that, by Theorem \ref{t:av} and by the composition rules for pseudodifferential operators, one has:
$$\frac{\tau_{\hbar}}{\hbar}[P_{\eps}(\hbar),\la B_{\hbar}\ra] = \ml{O}_{L^2 \rightarrow
L^2}(\tau_{\hbar}\hbar^{2}+\tau_{\hbar}\eps_{\hbar}^2).$$
In the case of a CROSS, we also note that the remainder is in fact of order $\ml{O}_{L^2 \rightarrow
L^2}(\tau_{\hbar}\eps_{\hbar}^2)$. If $\tau_{\hbar}\eps_{\hbar}^2\rightarrow 0^+$ (and $\tau_{\hbar}\hbar^2\rightarrow 0^+$ when $(M,g)$ is not a CROSS), after integrating
both sides of identity \eqref{e:wigeq} on the interval $t\in
[0,\tau]$ and taking limits as $\hbar\rightarrow 0^+$, one has for almost
every $\tau\in\IR$:
$$\mu(\tau)(b)=\mu(\tau)(\ml{I}_g(b))=\mu_0(\ml{I}_g(b)),$$
which, in view of \eqref{e:inv} concludes the proof of i) of the three Propositions.

We now turn to the proof of ii) of these different Propositions. Equation \eqref{e:comm-bis} implies that the commutator takes the form:
$$\frac{\tau_{\hbar}}{\hbar}[P_{\eps}(\hbar),\la B_{\hbar}\ra] = \tau_\hbar\hbar[\Oph(q_0),\la B_{\hbar}\ra]+\ml{O}_{L^2\rightarrow L^2}(\tau_{\hbar}\hbar^3)+\frac{\tau_{\hbar}\eps_{\hbar}^2}{\hbar}[V,\la B_{\hbar}\ra], $$
where $q_0$ was defined in paragraph~\ref{sss:commute} -- see also remark~\ref{r:zelditch} below. Integrating again
\eqref{e:wigeq} with respect to $t$, letting $\hbar\rightarrow
0^+$ and using the composition formula for pseudodifferential
operators gives that, for every $\tau\in\IR$ the following holds
\begin{equation}\label{e:trans}
\mu(\tau)(\ml{I}_g(b))-\mu_0(\ml{I}_g(b))=\int_0^{\tau}\mu(t)(\{L,\ml{I}_g(b)\})dt,
\end{equation}
where
\[ L:=
\left\{
\begin{array}[c]{lll}
q_0 & \text{if } \eps_{\hbar}\hbar^{-1}\rightarrow 0^+,\medskip \\
q_0+V & \text{if } \eps_{\hbar} = \hbar,\medskip \\
V & \text{if } \eps_{\hbar}\hbar^{-1}\rightarrow +\infty,
\end{array}
\right.
\]
In the case of a CROSS, we let $L:=V$ without any assumptions on the size of the perturbations. The geodesic flow preserves the symplectic form on $T^*M$,
therefore:
\begin{equation}\label{e:poiss}
\{\ml{I}_g(L),\ml{I}_g(b)\}=\ml{I}_g(\{\ml{I}_g(L),b\})=\ml{I}_g(\{L,\ml{I}_g(b)\}).
\end{equation}
In particular, combining this equality to~\eqref{e:inv}, we can rewrite~\eqref{e:trans} as
$$\forall\tau\in\IR,\ \mu(\tau)(\ml{I}_g(b))-\mu_0(\ml{I}_g(b))=\int_0^{\tau}\mu(t)(\{\ml{I}_g(L),\ml{I}_g(b)\})dt,$$
 Therefore, using again~\eqref{e:poiss}, one has $\frac{d}{d\tau}\mu(\tau)(\ml{I}_g(b\circ\varphi_L^{-\tau}))=0$, and thus
$$\mu(\tau)(\ml{I}_g(b)) = \mu_0(\ml{I}_g(b\circ\varphi^\tau_L))=\mu_0(\ml{I}_g(b)\circ\varphi^\tau_L), $$
where $\varphi^\tau_L$ is the Hamiltonian flow of $\ml{I}_g(L)$.
This completes the proof of ii) of Propositions~\ref{p:CROSS},~\ref{p:invariance}, and~\ref{p:invariancesmall}.

We consider now the large time regime, meaning $\tau_{\hbar}\eps_{\hbar}^2\rightarrow +\infty$ (in the case of Propositions~\ref{p:CROSS} and~\ref{p:invariance}) or $\tau_{\hbar}\hbar^2\rightarrow +\infty$ (in the case of Proposition~\ref{p:invariancesmall}). Let
$\theta$ be an element in $\ml{C}^1_c(\IR)$. We use an integration
by parts on \eqref{e:wigeq} to derive:
\begin{align*}
\int_{\IR}\theta'(t) & \left\la
u_{\hbar},e^{\frac{it\tau_{\hbar}P_{\eps}(\hbar)}{\hbar}} \la
B_{\hbar}\ra e^{-\frac{it\tau_{\hbar}P_{\eps}(\hbar)}{\hbar}}
u_{\hbar} \right\ra  dt=\\
&-\frac{i\tau_{\hbar}}{\hbar}\int_{\IR}\theta(t)\left\la
u_{\hbar},e^{\frac{it\tau_{\hbar}P_{\eps}(\hbar)}{\hbar}}
[P_{\eps}(\hbar),\la B_{\hbar}\ra]
e^{-\frac{it\tau_{\hbar}P_{\eps}(\hbar)}{\hbar}}
u_{\hbar}\right\ra dt.
\end{align*}
Thanks to relation~\eqref{e:comm-bis}, we deduce that
$$\frac{i}{\hbar}\int_{\IR}\theta(t)\left\la u_{\hbar},e^{\frac{it\tau_{\hbar}P_{\eps}(\hbar)}{\hbar}}
\left[\hbar^2\Oph(q_0)+\eps_{\hbar}^2V,\la B_{\hbar}\ra\right]
e^{-\frac{it\tau_{\hbar}P_{\eps}(\hbar)}{\hbar}}
u_{\hbar}\right\ra dt=\ml{O}(\tau_{\hbar}^{-1})+\ml{O}(\hbar^3).$$ 
Again, in the case of a CROSS, the remainder is only $\ml{O}(\tau_{\hbar}^{-1})$. Using the
composition formula for pseudodifferential operators, we find that
\begin{equation}\label{e:hinv}
\int_{\IR}\theta(t) \la
w_{\hbar}(t\tau_{\hbar}),\{\hbar^2q_0+\eps_{\hbar}^2V,\ml{I}_g(b)\}\ra
dt=\ml{O}(\tau_{\hbar}^{-1})+\ml{O}(\hbar^3).
\end{equation}
Taking limits as $\hbar \rightarrow 0^+$ and using identities
\eqref{e:inv}, \eqref{e:poiss} shows that
\[
\int_{\IR}\theta(t)\mu(t)(\{\ml{I}_g(L),\ml{I}_g(b)\})dt=0.
\]
This concludes the proof of part iii) in each case.

\begin{rema}\label{r:zelditch} Again, we would like to stress the fact that the symbol $q_0$ is only related to the choice of the metric on $M$. In order to get an (implicit) expression for $q_0$, we have to write that
$$0=[\la B_{\hbar}\ra,\hbar^2A^2]=[\la B_{\hbar}\ra,\tilde{P}_0(\hbar)^2-\hbar^2\Oph(q_0)]+\ml{O}_{L^2\rightarrow L^2}(\hbar^4).$$
Using energy cut-offs, this can be in fact rewritten as
$$\left[\la B_{\hbar}\ra,\Oph(q_0)\right]=
\frac{l^2}{2\pi^2\hbar^2}\left[\la
B_{\hbar}\ra,P_0(\hbar)\right]+\ml{O}_{L^2\rightarrow
L^2}(\hbar^2).$$
From this expression, we can get an expression for $\{\ml{I}_g(b),q_0\}$ by identifying the principal symbol of the right hand side which turns out to be a delicate task~\cite{Ze96, Ze97}. We will discuss this in more details in paragraph~\ref{ss:revolution} in the case of metrics of revolution on $\IS^2$.\\

\end{rema}

\subsection{Proof of Theorem~\ref{t:theo-zoll}}\label{ss:proof-theo-zoll}

Theorem \ref{t:theo-zoll} is now an easy consequence of Proposition \ref{p:invariance}. We will in fact prove something slightly more precise that works for any size of $\eps_{\hbar}$. We use the above convention for defining the function $L$. Recall that $L$ is a smooth function on $\mathring{T}^*M$ which is \emph{$0$-homogeneous} and that, in the case of a compact rank-one symmetric space, one has always $L= V\circ\pi$.

Let $\mu \in \ml{M}(\tau,\eps)$. Suppose $\tau_{\hbar}\eps_{\hbar}^2\rightarrow +\infty$, or $\tau_{\hbar}\hbar^2\rightarrow +\infty$ when $\eps_{\hbar}\ll\hbar$ (if $(M,g)$ is not a CROSS). Then, Propositions~\ref{p:CROSS},~\ref{p:invariance} and~\ref{p:invariancesmall} assert that $\mu(t,\cdot)$ is invariant
with respect to the flow $\varphi^{\tau}_L$ (the Hamiltonian flow
of $\ml{I}_g(L)$) for a.e. $t \in \IR$. 

First, we consider
$\Gamma_0\subset S^*M$ a closed orbit of $\varphi^s$ that is not included in 
$$\text{Crit}(L):=\left\{\rho\in\mathring{T}^*M : d_\rho\ml{I}_{g}(L)=0\right\}.$$ 
The fact that $\Gamma_0\not\subset\text{Crit}(L)$ is equivalent to the fact that the set $\Gamma_0$ is not invariant by $\varphi^{\tau}_L$ -- see Lemma \ref{l:inv} for a proof of this fact. For a given closed orbit $\Gamma_0\subset S^*M$ and for a fixed $\lambda>0$, we now define
$$\Gamma_0(\lambda):=\left\{(x,\lambda\xi):(x,\xi)\in\Gamma_0\right\}.$$
We fix a compact interval $[c_1,c_2]\subset(0,+\infty)$ and $\lambda\in[c_1,c_2]$. Since $\varphi^{\tau}_L$ commutes with the geodesic flow, it
follows that all of the sets $\varphi^{\tau}_L(\Gamma_0(\lambda))$
are distinct orbits of $\varphi^\tau$ (at least for $\tau$ small
enough). On the other hand, the invariance of $\mu(t,\cdot)$
implies that, for a.e. $t$ in $\IR$,
$$\mu(t)\left(\bigcup_{\lambda\in[c_1,c_2]}\varphi^{\tau}_L(\Gamma_0(\lambda))\right) =
\mu(t)\left(\bigcup_{\lambda\in[c_1,c_2]}\Gamma_0(\lambda)\right)$$ for every $\tau$ in some interval with
non-empty interior. Since $\mu(t,\cdot)$ is finite, one must have
\begin{equation}\label{e:lift-compact}\mu(t)\left(\bigcup_{\lambda\in[c_1,c_2]}\Gamma_0(\lambda)\right)=0,\end{equation}
for a.e. $t$ in $\IR$.

Consider now the projection $\nu(t,\cdot)$ of $\mu(t,\cdot)$ on $M$. Suppose $\Gamma_0\subset S^*M$ is a closed orbit of $\varphi^s$ and denote by $\gamma_0:=\pi(\Gamma_0)$ where $\pi:T^*M\rightarrow M$ is the canonical projection. Suppose that $\Gamma_0$ is not contained in $\text{Crit}(L)$. Using~\eqref{e:lift-compact}, we note that
$$\nu(t)(\gamma_0)=\mu(t)\left(\bigcup_{\lambda\in(0,\infty)}\left\{(x,\xi)\in\pi^{-1}(\gamma_0): \|\xi\|=\lambda\ \text{and}\ (x,\xi)\notin\Gamma_0(\lambda) \right\}\right).$$
Using invariance by the geodesic flow and the fact that the measure is finite, a similar argument as above allows to conclude that this quantity is in fact equal to $0$.

\subsection{Explicit expression of $q_0$ on $\IS^2$}\label{ss:revolution}

In this last section, we briefly recall a result due to Zelditch which gives a more or less explicit expression for $q_0$ -- see Theorem~$3$ in~\cite{Ze97} for more details. We suppose that $g$ is a $C_{2\pi}$-metric on $\mathbb{S}^2$; in this case, every geodesic is a simple curve, see \cite{GG81}. We fix $\Gamma_0$ a closed geodesic on $S^*M$ issued from the point $(x_0,\xi_0)$, i.e.
$$\Gamma_0:=\left\{\varphi^s(x_0,\xi_0)=(x(s),\xi(s):s\in\IR\right\}.$$
We denote by $\tilde{f}_{x_0,\xi_0}(t)$ the unique solution of the ordinary differential equation
$$f''(t)+K(x(t))f(t)=0,\ f(0)=1,\ f'(0)=0,$$
and by $f_{x_0,\xi_0}(t)$ the unique solution of the ordinary differential equation
$$f''(t)+K(x(t))f(t)=0,\ f(0)=0,\ f'(0)=1,$$
where $K$ is the scalar curvature. Using the conventions of Appendix \ref{s:geom-back}, we also define
$$\ml{K}(x,\xi):=g_x^*(d_xK,\xi^{\perp}).$$
The formula obtained by Zelditch can then be expressed as follows:
$$q_0(x_0,\xi_0):=\frac{1}{8\pi}\int_0^{2\pi}K(x(t))dt+\frac{1}{24\pi}\int_0^{2\pi}\ml{K}(x(t),\xi(t))\tilde{f}_{x_0,\xi_0}(t)^2R_{x_0,\xi_0}(t)dt.$$
where
\begin{align*}
R_{x_0,\xi_0}(t)&:=\tilde{f}_{x_0,\xi_0}(t)\int_0^t\ml{K}(x(s),\xi(s))f_{x_0,\xi_0}(s)^3ds\\&-3f_{x_0,\xi_0}(t)\int_0^t\tilde{f}_{x_0,\xi_0}(s)\ml{K}(x(s),\xi(s))f_{x_0,\xi_0}(s)^2ds.
\end{align*}
We proceed as in~\cite{Ze97} and specialize this expression to the case of $C_{2\pi}$-metrics of revolution on $\IS^2$. According to~\cite{Be78} -- Theorem~$4.13$, such a metric can be written in spherical coordinates as:
$$g=(1+\sigma(\cos\theta))^2d\theta^2+\sin^2\theta d\phi^2,$$
where $\sigma$ is a smooth odd function on $[-1,1]$ satisfying $\sigma(1)=0$. Recall from Paragraph~$4.17$ in~\cite{Be78} that the sectional curvature can be expressed as follows:
$$K(\theta)=\frac{1}{(1+\sigma(\cos\theta))^3}\left(1+\sigma(\cos\theta)-\cos\theta\sigma'(\cos\theta)\right).$$
Let $\pi(\Gamma_0)$ be the closed geodesic corresponding to $\theta=\frac{\pi}{2}$. We note that the curvature is constant and equal to $1$ on this closed geodesic. For every $(x_0,\xi_0)$ in $\Gamma_0$, we find then
$$q_0(x_0,\xi_0):=\frac{1}{4}+\frac{3^2\sigma'(0)^3}{4\pi}\int_0^{\pi}\left(\frac{\cos^6\phi}{3}-\cos^4\phi-\cos^2\phi\sin^4\phi\right)d\phi=\frac{1}{4}-\frac{3\sigma'(0)^3}{4},$$
In particular, $q_0$ is not equal to $\frac{1}{4}$ on this closed geodesic if $\sigma'(0)\neq 0$. This condition is of course generic among the possible choice of $\sigma$. Using the symmetry of revolution and the Gauss-Bonnet formula, one verifies that, for any $(x_0,\xi_0)$ belonging to a geodesic orthogonal to the geodesic $\theta=\frac{\pi}{2}$, one has $q_0(x_0,\xi_0)=\frac{1}{4}$. In particular, $q_0$ is not a constant function on $T^*\IS^2$. Combining this observation to the results of paragraph~\ref{ss:proof-theo-zoll}, we deduce Theorem~\ref{t:revolution} -- see Remark~\ref{r:eigen}.

\section{Measures with two invariance properties}\label{s:invariance}

This section is of more geometric flavor; it is devoted to analyzing the structure of measures on the cotangent bundle that are invariant by two flows that commute.

\subsection{Measures that are invariant by two Hamiltonians that commute}
As before, we write $p_0(x,\xi):=\frac{1}{2}||\xi||_x^2$ and $\varphi^s$ denotes the geodesic flow acting on $T^*M$.

Let $L$ be a smooth function on $\mathring{T}^*M$ satisfying:
\begin{itemize}
\item $L$ is $0$-homogeneous in the $\xi$ variable, i.e.
$L(x,\xi)=L(x,\lambda\xi)$ for every $\lambda>0$.
\item The one-form $dL$ does not vanish identically. 
\item $L$ Poisson-commutes with $p_0$: $\{L,p_0\}=0$.
\end{itemize}

We shall denote by $X_0$ and $X_L$ the respective Hamiltonian vector fields of $p_0$ and $L$; let $\varphi^{\tau}_{L}$ be the flow of $X_L$. Then the flows $\varphi^{\tau}_{L}$, and $\varphi^s$ commute. 

The purpose of this section is to describe the set of probability measures on $\mathring{T}^*M$ that are simultaneously invariant by the the flow $\varphi^{\tau}_{L}$ and by the geodesic flow $\varphi^s$. We are especially interested in the case when $d=2$ and $M$ is orientable,

We introduce the set $\text{Crit}(L)$ of critical points of $L$ in $\mathring{T}^*M$:
\[
\text{Crit}(L):= \left\{ \rho \in \mathring{T}^{*}M :  d_\rho L= 0 \right\}. 
\]
This is a closed set in $\mathring{T}^*M$ that is invariant by both flows $\varphi^s, \varphi^{\tau}_{L}$. More precisely, the following holds.
\begin{lemm}\label{l:inv}
Under the above assumptions, $\operatorname{Crit}(L)$ is formed by those orbits of the geodesic flow that are invariant by $\varphi^{\tau}_{L}$.
\end{lemm}
\begin{proof}
Since $\varphi^t$ and $\varphi^{s}_L$ commute, $\text{Crit}(L)$ is formed by orbits of the geodesic flow that are fixed by $\varphi^{s}_L$. On the other hand, if the orbit of the geodesic flow issued from some $\rho_0$ is invariant by $\varphi^{s}_L$ then necessarily $X_L(\rho_0)=\alpha X_0(\rho_0)$ for some $\alpha\in\IR$. This forces  $X_L(\rho)=\alpha X_0(\rho)$ for every $\rho$ in the orbit. Since $L$ is zero-homogeneous with respect to $\xi$ it turns out that $X_L$ and $X_0$ are orthogonal with respect to the Sasaki metric (see Appendix C). Therefore, $\alpha = 0$, as we wanted to prove.
\end{proof}

Consider now the set:
$$\ml{R}(L):=\mathring{T}^{*}M \setminus \text{Crit}(L).$$
This is an open subset of $\mathring{T}^{*}M$ that is invariant by $\varphi^s$ and $\varphi^{\tau}_{L}$. 
The map
$$\phi:\IR^2 \times  \ml{R}(L) \ni (s,t,\rho) \longmapsto \varphi^s \circ \varphi_L^t(\rho) \in \ml{R}(L),$$
is a group action of $\IR^2$ on $\ml{R}(L)$. Moreover, for any $\rho_0\in \ml{R}(L) $ the map:
$$\phi_{\rho_0}:\IR^2\ni (s,t) \longmapsto \varphi^s \circ \varphi_L^t(\rho_0) \in \ml{R}(L), $$
is an immersion. Therefore, the stabilizer group $G_{\rho_0}$ of $\rho_0$ under $\phi$ is discrete. This proves that the orbits of the action $\phi$ are either diffeomorphic to the torus $\IT^2$, to the cylinder $\IT\times \IR$ or to $\IR^2$.

The moment map:
$$\Phi : \ml{R}(L) \ni (x,\xi) \longmapsto (p_0(x,\xi), L(x,\xi)) \in \IR^2,$$
is a submersion, and for every $(E,J)\in\Phi(\ml{R}(L))$ the level set 
$$\Lambda_{(E,J)}:=\Phi^{-1}(E,J),$$
is a smooth submanifold of $\ml{R}(L)$ of codimension two. Note that the $0$-homogeneity of $L$ implies that $\Lambda_{(\alpha E,J)}=h_\alpha(\Lambda_{(E,J)})$ for every $\alpha>0$, and $h_\alpha:\mathring{T}^*M \longrightarrow \mathring{T}^*M$ being the homothety of ratio $\alpha$ on the fibers.

When $d=2$ the couple $p_0, L$ forms a completely integrable system on $\ml{R}(L)$; and the map $\phi_{\rho_0}$ induces a diffeomorphism:
$$\phi_{\rho_0}:\IR^2/G_{\rho_0}\longrightarrow \Lambda_{(E_0,J_0)}^{\rho_0},\quad \text{ for } (E_0,J_0):=\Phi(\rho_0);$$ 
Above, $\Lambda_{(E_0,J_0)}^{\rho_0}$ denotes the connected component of $\Lambda_{(E_0,J_0)}$ that contains $\rho_0$.
Therefore, if $\Lambda_{(E_0,J_0)}^{\rho_0}$ is compact then it is an embedded Lagrangian torus in $T^*M$. In that case, we shall write $\IT_{\rho_0}^2:=\IR^2 / G_{\rho_0}$.
See \cite{Duis80, MZ05} for more detailed proofs. 

Our next result clarifies the structure of the set of probability measures on $\ml{R}(L)$ that are invariant by the geodesic flow and the flow $\varphi_L^\tau$. Let us introduce first some notation: $\ml{R}_c(L)$ will be the set formed by those $\rho\in\ml{R}(L)$ such that $\Lambda_{\Phi(\rho)}^\rho$ is compact. 
\begin{prop}\label{p:biinv}
Suppose that $d=2$. Let $\mu$ be a probability measure on $\ml{R}(L)$ that is invariant by $\varphi^s$ and $\varphi_L^\tau$ and set $\overline{\mu}:=\Phi_*\mu$. Then, for every $b\in\ml{C}_c(\ml{R}(L))$, one has
$$ \int_{\ml{R}(L)}b(x,\xi)\mu(dx,d\xi)=\int_{\Phi(\ml{R}(L))}\int_{\Lambda_{(E,J)}}b(x,\xi)\lambda_{E,J}(dx,d\xi)\overline{\mu}(dE,dJ),$$
where, for $(E,J)\in\Phi(\ml{R}(L))$, the measure $\lambda_{E,J}$ is a convex combination of the (normalized) Haar measures on the tori $\Lambda_{(E,J)}^{\rho}$ for $\rho\in\Lambda_{(E,J)}\cap\ml{R}_c(L)$ (see equation \eqref{e:haar}).
\end{prop}

\begin{proof}
The disintegration Theorem (see e.g. Th.~$5.14$ in~\cite{EW11}) gives, for $\overline{\mu}$-a.e.~$(E,J)\in \ml{R}(L)$, the existence of a family of probability measures $\mu_{E,J}$ concentrated on $\Lambda_{(E,J)}$ such that:
$$ \int_{\ml{R}(L)}b(x,\xi)\mu(dx,d\xi)=\int_{\Phi(\ml{R}(L))}\int_{\Lambda_{(E,J)}}b(x,\xi)\mu_{E,J}(dx,d\xi)\overline{\mu}(dE,dJ),$$
for every $b\in\ml{C}_c(\ml{R}(L))$. The measures $\mu_{E,J}$ inherit the same invariance properties as $\mu$. In particular,
$$\int_{\Lambda_{(E,J)}}b(x,\xi)\mu_{E,J}(dx,d\xi)=\int_{\Lambda_{(E,J)}}b(\phi(s,t,x,\xi))\mu_{E,J}(dx,d\xi)$$
for every $(s,t)\in\IR^2$. Each connected component of the manifold $\Lambda_{(E,J)}$ has a group structure induced by the map $\phi_\rho$, for any $\rho\in \Lambda_{(E,J)}$. The invariance property of $\mu_{E,J}|_{\Lambda_{(E,J)}^{\rho}}$ is equivalent to stating that it is invariant by translations in the group. Therefore, $\mu_{E,J}|_{\Lambda_{(E,J)}^\rho}$ must coincide with a multiple of the Haar measure on $\Lambda_{(E,J)}^\rho$ for every $\rho\in\Lambda_{(E,J)}$. If $\Lambda_{(E,J)}^\rho$ is non compact, this measure is infinite and $\mu_{E,J}(\Lambda_{(H,J)}^\rho)=0$. 
\end{proof}
An explicit formula for the restriction of the measure $\lambda_{E,J}$ to a connected component $\Lambda_{(E,J)}^\rho$ with $\rho\in\ml{R}_c(L)\cap\Lambda_{(E,J)}$ is the following:
\begin{equation}\label{e:haar}
\int_{\Lambda_{(E,J)}^\rho}b(x,\xi)\lambda_{E,J}(dx,d\xi)=c\int_{\IT_{\rho_0}^2}b(\phi_\rho(s,t))dsdt,
\end{equation}
for some constant $c\in [0,1]$. Proposition \ref{p:biinv} merely states that bi-invariant measures a superpositions of measures of this form.

Note that so far we have not used the fact that the geodesic flow is periodic. When this is the case it turns out that for every $\rho_0\in\ml{R}_c (L)$ the stabilizer group $G_{\rho_0}$ of $\rho_0$ contains an element of the form $(q,0)$ for some $q\in\IR\setminus\{0\}$.

\subsection{Structure of the projection onto the base}
Our next goal is to study the regularity properties of the projection of a bi-invariant measure $\mu$ onto the base manifold $M$. We are going to prove that they decompose as an measure that is absolutely continuous with respect to the Riemannian volume plus a singular measure supported on the projection of the critical set $\text{Crit}(L)$.

\begin{theo}\label{t:abscon}
Suppose that $M$ is an orientable $P_l$-surface. Let $\mu$ be a probability measure on $\ml{R}(L)$ that is invariant by $\varphi^s$ and $\varphi_L^\tau$. Then $\nu:=\pi_*\mu$ is a probability measure on $M$ that is absolutely continuous with respect to the Riemannian measure. 
\end{theo}

Denote by $\ml{N}(L)$ the convex closure of the set of measures $\delta_{\pi\circ\gamma}$ where $\gamma\subset\mathring{T}^*M$ ranges over the orbits of the geodesic flow that are contained in $\text{Crit}(L)$. The following is a direct consequence of Theorem \ref{t:abscon}.
\begin{coro}\label{c:struct}
Suppose that $M$ is an orientable $P_l$-surface. The projection $\nu:=\pi_*\mu$ of a probability measure $\mu$ on $\mathring{T}^*M$ that is invariant by $\varphi^s$ and $\varphi_L^\tau$ can be decomposed as:
$$\nu = f \operatorname{vol}_g + \alpha \nu_{\emph{sing}}$$
where $f\in L^1(M)$, $\alpha\in[0,1]$ and $\nu_{\emph{sing}}\in\ml{N}(L)$.
\end{coro}

\begin{rema}
For a ``generic" function $L$, the set $\text{Crit}(L)$ is formed by critical points that are non degenerate with respect to directions that are orthogonal to the geodesic flow and to the vertical vector field that generates the dilations on the fibers. In this case, for every $E>0$ the set $\text{Crit}(L)\cap p_0^{-1}(E)$ consists in a finite number of orbits of the geodesic flow that project onto a finite family of geodesics $\{\gamma_1,...,\gamma_r\}$ in $M$. The measures $\nu$ described in Corollary \ref{c:struct} then take the simple form:
$$\nu = f \operatorname{vol}_g+\sum_{j=1}^r c_j \delta_{\gamma_j},$$ 
where $c_j\geq 0$ for $j=1,...,r$ and $\sum_{j=1}^r c_j \leq 1$.
\end{rema}

Theorem \ref{t:abscon} essentially follows from a structure result for the projections of the Lagrangian tori $\Lambda_{(E,J)}^{\rho_0}$ onto $M$: the projection $\pi|_{\Lambda_{(E,J)}^{\rho_0}}$ is a local diffeomorphism onto its image except at points lying on a finite family of smooth curves (called sometimes a \emph{caustic}). See also Lemma 2.1 in \cite{BiPo89}.

\begin{lemm}\label{l:caus}
Let $\rho_0\in\ml{R}_c(L)$; then there exist a finite family $\{\Gamma_{\rho_0}^j\}_{j=1}^n$ of smooth disjoint closed curves contained in $\Lambda_{(E,J)}^{\rho_0}$ that are traversal to the geodesic flow and such that $\Gamma_{\rho_0}^{\text{caus}}:=\bigsqcup_{j=1}^n \Gamma_{\rho_0}^j$ is exactly the set of singular points of $\pi|_{\Lambda_{(E,J)}^{\rho_0}}:\Lambda_{(E,J)}^{\rho_0}\longrightarrow M$.
\end{lemm}
\begin{proof}
Let $\{X_0,W,Y_0,U\}$ be the orthogonal frame on $\mathring{T}^*M$ defined in Appendix C. Since $L$ is $0$-homogeneous with respect to $\xi$, the Hamiltonian vector field $X_L$ has no component in $X_0$, that is:
$$g_{\rho}^S(X_L(\rho), X_0(\rho))=g_{\pi(\rho)}(d_{\rho}\pi (X_L(\rho)), d_{\rho}\pi (X_0(\rho)))=0.$$
The tangent space of $\Lambda_{(E,J)}^{\rho_0}$ at a point $\rho$ is spanned by $\{X_0(\rho), X_L(\rho)\}$. The points $\rho\in \Lambda_{(E,J)}^{\rho_0}$ close to which $\pi|_{\Lambda_{(E,J)}^{\rho_0}}$ is not a local diffeomorphism are precisely those at which $d_{\rho}\pi (X_L(\rho))=0$.
Define:
$$F(t,s):=g_{\pi(\phi_{\rho_0}(t,s))}(d_{\phi_{\rho_0}(t,s)}\pi (X_L(\phi_{\rho_0}(t,s))), d_{\phi_{\rho_0}(t,s)}\pi (W(\phi_{\rho_0}(t,s)))).$$
One has:
$$F(t,s)=0 \Longleftrightarrow X_L(\phi_{\rho_0}(t,s)) \in \text{Ker}(d_{\phi_{\rho_0}(t,s)}\pi).$$
We are going to apply the implicit function Theorem to the equation $F(t,s)=0$ in order to prove our claim. 
Note that $J_s(t):=d_{\phi_{\rho_0}(t,s)}\pi (X_L(\phi_{\rho_0}(t,s)))=\partial_s(\pi\circ\phi_{\rho_0})(t,s)$ is a Jacobi field along the geodesic $\gamma_s(t):=\pi(\varphi^t(\varphi_L^s(\rho_0))$. Write $Z_s(t):=d_{\phi_{\rho_0}(t,s)}\pi (W(\phi_{\rho_0}(t,s)))$. One has: 
$$\partial_t F(t,s)= g_{\pi(\phi_{\rho_0}(t,s))}\left(\frac{DJ_s}{dt}(t), Z_s(t)\right)
							+g_{\pi(\phi_{\rho_0}(t,s))}\left(J_s(t), \frac{DZ_s(t)}{dt}\right),$$
where $\frac{D}{dt}$ is the covariant derivative along the curve $\gamma_s(t)$. If $F(t_0,s_0)=0$ then $J_{s_0}(t_0)=0$ but:
$$g_{\pi(\phi_{\rho_0}(t_0,s_0))}\left(\frac{DJ_{s_0}}{dt}(t_0), Z_{s_0}(t_0)\right)\neq 0;$$
the reason for this is that if that term were to vanish, this would imply that $\frac{D}{dt}J_{s_0}(t_0)$ is proportional to $\gamma '(t_0)$. Therefore, the Jacobi field $J_{s_0}(t)$ should be proportional to $t\gamma '(t)$ which forces $J_{s_0}$ to vanish identically. In other words, $X_L(\phi_{\rho_0}(t,s_0))\in\text{Ker}( d_{\phi_{\rho_0}(t,s_0)}\pi)$ for every $t$. But this is a contradiction, since the values of $t$ for which that property holds turn out to form a discrete set in $\IR$, as we show next. 

This fact is a consequence of the twist property of the vertical subbundle as stated in \cite{Pa99}: recall that $\Lambda_{(E,J)}^{\rho_0}$ is a Lagrangian submanifold of $T^*M$; applying Proposition 2.11 of \cite{Pa99} to the Lagrangian subspaces $T_{\rho}\Lambda_{(E,J)}^{\rho_0}$ with $\rho\in\Lambda_{(E,J)}^{\rho_0}$, we deduce that, for every $s_0$ the set of the $t\in\IR$ such that: 
$$\text{Ker}( d_{\phi_{\rho_0}(t,s_0)}\pi)\cap T_{\phi_{\rho_0}(t,s_0)}\Lambda_{(L,J)}^{\rho_0}=\text{Ker}( d_{\phi_{\rho_0}(t,s_0)}\pi)\cap d_{\varphi_L^{s_0}(\rho_0)}\varphi_{0}^t(T_{\varphi_L^{s_0}(\rho_0)}\Lambda_{(E,J)}^{\rho_0})\neq\{0\},$$
is discrete. 

If the geodesic $\varphi^t(\rho_0)$ is $q$-periodic, this set is in fact finite modulo $q\IZ$. Set $s_0=0$ and let $\{t_1,...,t_r\}\subset\IR/q\IZ$ be the set at which $F(t,0)=0$. We can apply the implicit function Theorem to conclude the existence of a unique family of smooth functions $t_j(s)$ with $s\in\IR/r(J)\IZ$ satisfying $t_j(0)=t_j$ and $F(t_j(s),s)=0$. Taking into account that $\partial_t F(t,s)\neq 0$ as soon as $F(t,s)=0$, we conclude that two curves $(t_j(s),s)$ and $(t_i(s),s)$ are either disjoint or they coincide. Let us relabel these curves in order to have $(t_j(s),s)$ with $j=1,...n$ are disjoint. Since these curves are the unique solutions to the equation $F(t,s)=0$, the trajectories $(t_j(s),s)$ must be closed. Therefore, the smooth curves $\Gamma_{\rho_0}^j:=\{\phi_{\rho_0}(t_j(s),s):s\in\IR/r(J)\IZ\}$ with $j=1,...,n$ are disjoint and closed. 
\end{proof}

To conclude the proof of Theorem \ref{t:abscon} let $B\subset M$ be of zero Riemannian measure. By construction,
$$\nu(B) = \int_{\Phi(\ml{R})(L)}\lambda_{E,J}(\pi^{-1}(B)\cap \Lambda_{(E,J)}) \overline{\mu}(dE,dJ).$$
Fix a connected component $\Lambda_{(E,J)}^{\rho_0}$ in $\Lambda_{(E,J)}$, let $\Omega_{\rho_0}$ be any connected component of $\Lambda_{(E,J)}^{\rho_0}\setminus \Gamma_{\rho_0}^{\text{caus}}$. Then $\lambda_{E,J}(\Gamma_{\rho_0}^{\text{caus}})=0$ by Lemma \ref{l:caus} and $\pi |_{\Omega_{\rho_0}}$ is a local diffeomorphism; therefore, $\pi^{-1}(B)\cap\Omega_{\rho_0}$ has $\lambda_{E,J}$-measure equal to zero. Hence,
$$\lambda_{E,J}(\pi^{-1}(B)\cap \Lambda_{(H,J)}^{\rho_0}) =0. $$
This shows that $\nu(E)=0$ and therefore, $\nu$ is absolutely continuous with respect to Riemannian mesure.

\section{Quantum Loschmidt echo on Zoll manifolds}

In this section, we revisit the proof of section~\ref{s:averaging} in order to study a quantity which is related to the so called quantum Loschmidt echo defined in the section~\ref{s:main}. Precisely, we prove Proposition~\ref{t:loschmidt}. For every $b$ in $\mathcal{C}^{\infty}_c(T^*M)$ and every $t$ in $\IR$, we define
$$M_{\hbar}(t\tau_{\hbar})(b):=\langle v_{\hbar}(t\tau_{\hbar}),\Oph(b) v_{\hbar}^0(t\tau_{\hbar})\rangle,$$
where 
\begin{itemize}
\item $(\tau_{\hbar})_{0<\hbar\leq 1}$ is a scale of times satisfying $\lim_{\hbar\rightarrow 0}\tau_{\hbar}=+\infty$;
\item $v_{\hbar}(t\tau_{\hbar})$ is the solution of~\eqref{e:perturbed-schr};
\item $v_{\hbar}^0(t\tau_{\hbar})$ is the solution of~\eqref{e:perturbed-schr} when we pick $V\equiv 0$.
\end{itemize}
If we suppose that the sequence of initial data $(u_{\hbar})_{0<\hbar\leq 1}$ is normalized and that it satisfies~\eqref{e:hosc} and~\eqref{e:shosc}, then, proceeding as in appendix~\eqref{a:measures}, we can extract a subsequence $\hbar_n\rightarrow 0$ such that, for every $\theta\in\ml{C}^{\infty}_c(\IR)$ and for every $a$ in $\mathcal{C}^{\infty}_c(T^*M)$, one has
$$\lim_{n\rightarrow +\infty}\int_{\IR}\theta(t)M_{\hbar_n}(t\tau_{\hbar_n})(b)dt=\int_{\IR}\theta(t)\left(\int_{T^*M} b(x,\xi) M(t,dx,d\xi)\right)dt,$$
where $t\mapsto M(t)$ belongs to $L^{\infty}(\IR,\ml{M}(T^*M))$. Thanks to the frequency assumptions, the support of $M(t)$ is included in $T^*M\backslash (M\times\{0\}).$ Moreover, in the case where $\eps_{\hbar}\ll\sqrt{\hbar}$, we also observe that, for a.e. $t$ in $\IR$ and for every $s$ in $\IR$, one has
\begin{equation}\label{e:invariance-complex}\int_{T^*M} b(x,\xi) M(t,dx,d\xi)=\int_{T^*M} b\circ \varphi^s(x,\xi) M(t,dx,d\xi).\end {equation}
\begin{rema}
In order to make the link with Proposition~\ref{t:loschmidt}, we note that, thanks to~\eqref{e:hosc} and~\eqref{e:shosc}, one has
 $$\lim_{n\rightarrow +\infty}\int_{\IR}\theta(t)M_{\hbar_n}(t\tau_{\hbar_n})(1)dt=\int_{\IR}\theta(t)\left(\int_{T^*M}  M(t,dx,d\xi)\right)dt.$$
\end{rema}
Proposition~\ref{t:loschmidt} follows then from the previous remark and the following Proposition.
\begin{prop}\label{p:loschmidt} Suppose $(M,g)$ is a $P_l$-manifold and that
$$\lim_{\hbar\rightarrow 0^+}\eps_{\hbar}\hbar^{-\frac{1}{2}}=0.$$
Suppose also that one the following condition holds:
\begin{enumerate}
 \item $\eps_{\hbar}\hbar^{-\frac{3}{2}}\rightarrow +\infty$,
 \item $(M,g)$ is a Compact Rank One Symmetric Space.
\end{enumerate}
Then, for every $t\mapsto F(t)$ associated to a sequence of initial data with an unique semiclassical measure $\mu_0$, the following holds:

\medskip
\noindent i) If $\lim_{\hbar\rightarrow 0^+}\frac{\tau_{\hbar}\eps_{\hbar}^2}{\hbar}=0$, then, for every $t$ in $\IR$ and for every $b$ in $\ml{C}^{\infty}_c(T^*M)$,
$$M(t)(b)=\mu_0(\ml{I}_g(b)).$$
\medskip
\noindent ii)  If $\tau_{\hbar}=\frac{\hbar}{\eps_{\hbar}^2}$, then, for every $t$ in $\IR$ and for every $b$ in $\ml{C}^{\infty}_c(T^*M)$,
$$M(t)(b)=\mu_0\left(\ml{I}_g(b)e^{it\ml{I}_g(V)}\right).$$
\medskip
\noindent iii) If $\lim_{\hbar\rightarrow 0^+}\frac{\tau_{\hbar}\eps_{\hbar}^2}{\hbar}=+\infty$, then, for a.e. every $t$ in $\IR$ and for every $b$ in $\ml{C}^{\infty}_c(T^*M)$,
$$M(t)(b\ml{I}_g(V))=0.$$
\end{prop}

\begin{rema} In the case (iii), we emphasize that we can deduce that 
$$\text{supp} M(t)\subset\{\ml{I}_g(V)=0\}.$$ 
\end{rema}

We will now give the proof of this Proposition which follows the same lines as the proofs in section~\ref{s:averaging}.

\begin{proof} Let $b$ be a smooth compactly supported function on $\mathring{T}^{*}M$. Using~\eqref{e:invariance-complex}, one has that
$$M(t)(b)=M(t)(\ml{I}_g(b)).$$
We start our proof by computing the following derivative
$$\frac{d}{dt}\left(\langle v_{\hbar}(t\tau_{\hbar}),\langle B_{\hbar}\rangle v_{\hbar}^0(t\tau_{\hbar})\rangle\right),$$
where $\langle B_{\hbar}\rangle$ is the operator from Theorem~\ref{t:av}. Using Theorem~\ref{t:av}, we obtain in fact
\begin{equation}\label{e:der-loschmidt}\frac{d}{dt}\left(\langle v_{\hbar}(t\tau_{\hbar}),\langle B_{\hbar}\rangle v_{\hbar}^0(t\tau_{\hbar})\rangle\right)=-\frac{i\eps_{\hbar}^2\tau_{\hbar}}{\hbar}\langle v_{\hbar}(t\tau_{\hbar}),V\langle B_{\hbar}\rangle v_{\hbar}^0(t\tau_{\hbar})\rangle+\ml{O}(\tau_{\hbar}\hbar^2).\end{equation}
We observe that, if $(M,g)$ is a compact rank-one symmetric space, then the remainder is in fact equal to $0$. The case (i) follows immediatly by integrating~\eqref{e:der-loschmidt} between $0$ and $t$. In the case where $\tau_{\hbar}=\frac{\hbar}{\eps_{\hbar}^2}$, one finds that, for every $\theta$ in $\ml{C}^{\infty}_c(\IR)$, one has 
$$\int_{\IR}\theta'(t)M(t)(\ml{I}_g(b))dt=i\int_{\IR}\theta(t)M(t)(\ml{I}_g(b)V)dt=i\int_{\IR}\theta(t)M(t)(\ml{I}_g(b)\ml{I}_g(V))dt,$$
where the first equality follows from~\eqref{e:der-loschmidt} and the second one from~\eqref{e:invariance-complex}. The case (ii) follows then from the above relation. Finally, when $\lim_{\hbar\rightarrow 0^+}\frac{\tau_{\hbar}\eps_{\hbar}^2}{\hbar}=+\infty$, we deduce from~\eqref{e:der-loschmidt} that
$$\int_{\IR}\theta(t)M(t)(\ml{I}_g(b)\ml{I}_g(V))dt=0,$$
from which (iii) follows according to~\eqref{e:invariance-complex}.
\end{proof}

\section{Relation to spectral problems: proof of Proposition~\ref{p:level-spacing}}
\label{s:level-spacing}

In this last section, we explain how time dependent semiclassical measures can be used to prove some results on the spectrum of the semiclassical Schr\"odinger operator 
$$P_{\eps}(\hbar):=-\frac{\hbar^2\Delta_g}{2}+\eps_{\hbar}^2V.$$
Precisely, we prove Proposition~\ref{p:level-spacing} from the introduction which establishes the existence of distinct eigenvalues which are asymptotically close to each other. In order to prove this result, we revisit an argument that was
given in~\cite{Ma08} -- proof of Theorem~$2$ from this reference. Thanks to Corollary~\ref{c:zoll-coro}, recall that any weak-$\star$ accumulation point $\nu$ of the sequence $(|\psi_{\hbar}|^2\text{vol}_g)_{0<\hbar\leq 1}$ verifying
$$P_{\eps}(\hbar)\psi_{\hbar}=E(\hbar)\psi_{\hbar},\quad \text{and}\ E(\hbar)\rightarrow E_0\neq 0,$$
must satisfy $$\nu(\{\gamma_0\})=0,$$
for every closed geodesic $\gamma_0$ that is not included in $\ml{C}(V)$. Let $\gamma_0$ be such a geodesic; let $\Gamma_0$ be a lift of $\gamma_0$ to $\mathring{T}^*M$ such that, for $(x_0,\xi_0)\in\Gamma_0$ one has
$\frac{\|\xi_0\|^2}{2}=E_0$ . We proceed by
contradiction and we suppose that there exists $\hbar_n\rightarrow
0$ as $n\rightarrow +\infty$ such that
$$\tilde{s}_0(\hbar_n):=\frac{\inf\{|E_j(\hbar_n)-E_k(\hbar_n)|:1\leq j\neq k \leq N(\hbar_n)\}}{\hbar_n\eps_{\hbar_n}^2}\rightarrow +\infty.$$
Our goal is to construct $\nu$ as above verifying
$\nu(\{\gamma_0\})>0$.

We fix a sequence $\tau_{\hbar_n}\rightarrow+\infty$ satisfying
\begin{equation}\label{e:hyp-time-1}\tau_{\hbar_n} \tilde{s}_0(\hbar_n)\eps_{\hbar_n}^2\rightarrow +\infty,\end{equation}
and
\begin{equation}\label{e:hyp-time-2}\tau_{\hbar_n}\eps_{\hbar_n}^2\rightarrow 0.\end{equation}
These two conditions are compatible thanks to our assumption on
the subsequence $\hbar_n\rightarrow 0^+$. In the following, we
omit the subscript $n$ in order to alleviate the notations.

\subsection{A sequence of ``good'' initial data}

We fix a sequence of initial data
$(u_{x_0,\xi_0}^{\hbar})_{0<\hbar\leq 1}$ whose semiclassical
measure is unique and given by the Dirac measure
$\delta_{x_0,\xi_0}$. This can be obtained from a sequence of
normalized coherent states -- see for instance Ch.~$5$
in~\cite{Zw12}. Then, we pick a cutoff function $0\leq \chi\leq 1$
in $\ml{C}^{\infty}_c(\IR)$ which is equal to $1$ in a small
neighborhood of $0$, say $[-1/2,1/2]$, and $0$ outside a slightly
larger neighborhood, say $[-1,1]$. We introduce the following
truncation of our coherent states:
$$\tilde{u}_{x_0,\xi_0}^{\hbar}:=\chi\left(\frac{P_0(\hbar)-E_0}{\delta_0}\right)u_{x_0,\xi_0}^{\hbar}.$$
Using results on functional calculus for semiclassical
pseudodifferential operators -- see for instance Ch.~$14$
in~\cite{Zw12}, one can verify that the operator
$\chi\left(\frac{P_0(\hbar)-E_0}{\delta_0}\right)$ is an
$\hbar$-pseudodifferential operator belonging to the set
$\Psi^{-\infty,0}(M)$ as defined in Appendix \ref{a:pdo}. Recall that its
principal symbol is
$$\chi_{\hbar}^0(x,\xi):=\chi\left(\frac{\|\xi\|^2/2-E_0}{\delta_0}\right).$$
As the semiclassical measure of the sequence
$(u_{x_0,\xi_0}^{\hbar})_{0<\hbar\leq 1}$ is the Dirac measure in
$(x_0,\xi_0)$, one can observe that
\begin{equation}\label{e:truncated}\|u_{x_0,\xi_0}^{\hbar}-\tilde{u}_{x_0,\xi_0}^{\hbar}\|_{L^2}=o(1).\end{equation}
Thanks to the Calder\'on-Vaillancourt Theorem, one can also verify
that
\begin{equation}\label{e:taylor}\left\|\chi\left(\frac{P_0(\hbar)-E_0}{\delta_0}\right)-\chi\left(\frac{P_{\eps}(\hbar)-E_0}{\delta_0}\right)\right\|_{L^2(M)\rightarrow L^2(M)}=o(1).\end{equation}
Denote by $\Pi_{E_j(\hbar)}$ the spectral projector corresponding
to the eigenvalue $E_j(\hbar)$, one has
$$\tilde{u}_{x_0,\xi_0}^{\hbar,\eps}:=\chi\left(\frac{P_{\eps}(\hbar)-E_0}{\delta_0}\right)u_{x_0,\xi_0}^{\hbar}=\sum_{j=1}^{N(\hbar)}\chi\left(\frac{E_j(\hbar)-E_0}{\delta_0}\right)\Pi_{E_j(\hbar)}u_{x_0,\xi_0}^{\hbar}.$$
Recall that the eigenvalues $E_j(\hbar)$ (and the spectral
projectors) depend implicitly on $\eps_{\hbar}$. When
$\Pi_{E_j(\hbar)}u_{x_0,\xi_0}^{\hbar}\neq 0$, we define
$$v_{\hbar,\eps}^j:=\frac{\Pi_{E_j(\hbar)}u_{x_0,\xi_0}^{\hbar}}{\|\Pi_{E_j(\hbar)}u_{x_0,\xi_0}^{\hbar}\|_{L^2}},$$
otherwise, we set $v_{\hbar,\eps}^j=0$. Then, we write
$$\tilde{u}_{x_0,\xi_0}^{\hbar,\eps}=\sum_{j=1}^{N(\hbar)}\chi\left(\frac{E_j(\hbar)-E_0}{\delta_0}\right)\left\|\Pi_{E_j(\hbar)}u_{x_0,\xi_0}^{\hbar}\right\|_{L^2}v_{\hbar,\eps}^j.$$
Finally, we introduce
$c_{\hbar,\eps}^j:=\chi\left(\frac{E_j(\hbar)-E_0}{\delta_0}\right)^2\left\|\Pi_{E_j(\hbar)}u_{x_0,\xi_0}^{\hbar}\right\|_{L^2}^2$.
Note that this quantity is equal to $0$ when $v_{\hbar,\eps}^j=0$. Using~\eqref{e:truncated} and~\eqref{e:taylor}, one can
verify that
\begin{equation}\label{e:convex-sum}\sum_{j=1}^{N(\hbar)}c_{\hbar,\eps}^j=1+o(1).\end{equation}

\subsection{Semiclassical measures of the Schr\"odinger equation}

Let $\theta$ be a smooth function on $\IR$ such that
$\hat{\theta}$ is compactly supported and satisfies
$\hat{\theta}(0)=1.$ We also fix $b$ in $\ml{C}^{\infty}_c(\mathring{T}^*M)$.
According to Propositions~\ref{p:CROSS} and~\ref{p:invariance} (point i) and to the fact that
$\tau_{\hbar}\eps_{\hbar}^{2}\rightarrow 0$, one has
$$\int_{\IR}\theta(t)\left\la  u_{x_0,\xi_0}^{\hbar},e^{\frac{it\tau_{\hbar}P_{\eps}(\hbar)}{\hbar}}\Oph(b)e^{\frac{it\tau_{\hbar}P_{\eps}(\hbar)}{\hbar}}u_{x_0,\xi_0}^{\hbar}\right\ra dt=\ml{I}_g(b)(x_0,\xi_0)+o(1),$$
as $\hbar$ goes to $0$. Using the conventions introduced in the
previous paragraph, one deduces then
$$\sum_{1\leq j,k\leq N(\hbar)}\hat{\theta}\left(\frac{\tau_{\hbar}(E_j(\hbar)-E_k(\hbar)}{\hbar}\right) \left(c_{\hbar,\eps}^jc_{\hbar,\eps}^k\right)^{\frac{1}{2}}\left\la v_{\hbar,\eps}^j,\Oph(b)v_{\hbar,\eps}^k\right\ra =\ml{I}_g(b)(x_0,\xi_0)+o(1).$$
For $j\neq k$, one has
$$\left|\frac{\tau_{\hbar}(E_j(\hbar)-E_k(\hbar))}{\hbar}\right|\geq \tau_{\hbar}\tilde{s}_0(\hbar)\eps_{\hbar}^{2},$$
which tends to $+\infty$ according to~\eqref{e:hyp-time-1}. Thus,
we get, as $\hbar$ goes to $0$,
\begin{equation}\label{e:convex-sum-measure}\sum_{j=1}^{N(\hbar)} c_{\hbar,\eps}^j\left\la v_{\hbar,\eps}^j,\Oph(b)v_{\hbar,\eps}^j\right\ra =\ml{I}_g(b)(x_0,\xi_0)+o(1).\end{equation}
Recall that $v_{\hbar,\eps}^j$ is either $0$ or a normalized state in $L^2(M)$
which satisfies
$P_{\eps}(\hbar)v_{\hbar,\eps}^j=E_j(\hbar)v_{\hbar,\eps}^j.$ We
will deduce from this asymptotic formula the existence of an
accumulation element $\nu$ satisfying
$\nu(\gamma_0)=1$, and thus we will obtain the expected
contradiction as $\gamma_0$ is not cointained in $\ml{C}(V)$.

\subsection{Estimating the variance}

We now remark that the measure $b\mapsto\ml{I}_g(b)(x_0,\xi_0)$ is
an ergodic measure which can be written as a convex sum of almost
invariant and almost positive distribution. Thus, we can proceed
as in the proof of the quantum ergodicity Theorem~\cite{Zw12} to
construct a subsequence converging to this measure. For that
purpose, we start by estimating the variance:
$$V_{\hbar}(b):=\sum_{j=1}^{N(\hbar)} c_{\hbar,\eps}^j\left|\left\la v_{\hbar,\eps}^j,\Oph(b)v_{\hbar,\eps}^j\right\ra -\ml{I}_g(b)(x_0,\xi_0)\right|^2.$$
One can introduce $\chi_1$ a smooth compactly supported function
on $\mathring{T}^*M$ which is equal to $1$ for
$\frac{E_0}{2}\leq\|\xi\|^2\leq 8 E_0$ and which is invariant by
the geodesic flow. By construction of the states $v_{\hbar,\eps}^j$, we can rewrite $V_{\hbar}(b)$ as
follows:
$$V_{\hbar}(b)=\sum_{j=1}^{N(\hbar)} c_{\hbar,\eps}^j\left|\left\la v_{\hbar,\eps}^j,\Oph(\tilde{b}\chi_1)v_{\hbar,\eps}^j\right\ra \right|^2+\ml{O}(\hbar),$$
where $\tilde{b}=b-\ml{I}_g(b)(x_0,\xi_0).$ Combining the Egorov
Theorem to the fact that the $v_{\hbar,\eps}^j$ are eigenmodes,
one finds, for every $T>0$,
$$V_{\hbar}(b)=\sum_{j=1}^{N(\hbar)} c_{\hbar,\eps}^j\left|\left\la v_{\hbar,\eps}^j,\Oph\left(\frac{1}{T}\int_0^T(\tilde{b}\chi_1)\circ\varphi^s ds\right)v_{\hbar,\eps}^j\right\ra \right|^2+\ml{O}(\hbar),$$
where the constant in the remainder depends on $T$. Applying the
Cauchy-Schwarz inequality and the composition rule for
pseudodifferential operators, we obtain
$$V_{\hbar}(b)\leq\sum_{j=1}^{N(\hbar)} c_{\hbar,\eps}^j\left\la v_{\hbar,\eps}^j,\Oph\left(\left|\frac{1}{T}\int_0^T(\tilde{b}\chi_1)\circ\varphi^s ds\right|^2\right)v_{\hbar,\eps}^j\right\ra +\ml{O}(\hbar),$$
We now use the limit formula~\eqref{e:convex-sum-measure}, and we
derive that, for every $T>0$,
$$\limsup_{\hbar\rightarrow 0}V_{\hbar}(b)\leq\ml{I}_g\left(\left|\frac{1}{T}\int_0^T(\tilde{b}\chi_1)\circ\varphi^s ds\right|^2\right)(x_0,\xi_0).$$
We take the limit $T\rightarrow +\infty$, and we have, as
$\hbar\rightarrow 0^+$,
\begin{equation}\label{e:variance-estimate}
\sum_{j=1}^{N(\hbar)} c_{\hbar,\eps}^j\left|\left\la
v_{\hbar,\eps}^j,\Oph(b)v_{\hbar,\eps}^j\right\ra
-\ml{I}_g(b)(x_0,\xi_0)\right|^2=o(1).
\end{equation}

\subsection{Bienaym\'e-Tchebychev inequality}
As the eigenmodes are microlocalized on the energy layers
$E_0-\delta_0\leq \frac{\|\xi\|_x^2}{2}\leq E_0+\delta_0$, we fix
$(b_k)_{k\in\IN}$ a family of smooth functions in
$\ml{C}_c^{\infty}(\mathring{T}^*M)$ which defines by restriction a dense
subset (in the $\ml{C}^0$ topology) of the continuous functions on
$\{(x,\xi)\in T^*M: [E_0-2\delta_0,E_0+2\delta_0]\}$. We
reestablish the dependence in the parameter $n$ and we define, for
every $k$ in $\IN$ and every $n$ in $\IN$,
$$A_{\hbar_n}^k:=\left\{1\leq j\leq N(\hbar_n):\left|\left\la v_{\hbar_n,\eps}^j,\Op_{\hbar_n}(b_k)v_{\hbar_n,\eps}^j\right\ra -\ml{I}_g(b_k)(x_0,\xi_0)\right|^2\geq V_{\hbar_n}(a_k)^{\frac{1}{2}}\right\}.$$
From the Bienaym\'e-Tchebychev inequality, we get, as
$n\rightarrow+\infty$,
$$\sum_{j\in A_{\hbar_n}^k} c_{\hbar_n,\eps}^j=o(1).$$
For every $k\geq 0$, we set
$\mathbf{A}_{\hbar_n}^{k}:=\cup_{l=0}^k A_{\hbar_n}^{l}$ which
still satisfies
$$\sum_{j\in \mathbf{A}_{\hbar_n}^k} c_{\hbar_n,\eps}^j=o(1).$$
For every $k\in \IN$, we take $n_k$ larke enough to ensure that,
for every $n\geq k$, one has
$$\left|\sum_{j\in \mathbf{A}_{\hbar_n}^k} c_{\hbar_n,\eps}^j\right|\leq\frac{1}{k},$$
and we also impose the subsequence $n_k$ to be increasing. Then,
for every $n_k\leq n<n_{k+1}$, we let $\mathbf{A}_{\hbar_n}:=
\mathbf{A}_{\hbar_n}^k.$ Then, thanks to~\eqref{e:convex-sum}, one
has, as $n\rightarrow+\infty$,
$$\sum_{j\in \mathbf{A}_{\hbar_n}^c} c_{\hbar_n,\eps}^j=1+o(1).$$
By construction, this implies that there exists a sequence $1\leq
j_n\leq N(\hbar_n)$ such that, for every $k$ in $\IN$,
$$\lim_{n\rightarrow +\infty}\la v_{\hbar_n,\eps}^{j_n},\Op_{\hbar_n}(b_k)v_{\hbar_n,\eps}^{j_n}\ra =\ml{I}_g(b_k)(x_0,\xi_0).$$
By density of the family $b_k$ in the $\ml{C}^0$-topology, we find
that the limit measure is the measure carried by the closed
geodesic issued from $(x_0,\xi_0)$. In other words, we have
constructed a sequence of eigenmodes whose
semiclassical measure is carried by the closed geodesic issued
from $(x_0,\xi_0)$. As we have supposed that this closed geodesic is not included in $\text{Crit}(V)$, we obtain the
contradiction.

\appendix
\section{Semiclassical analysis on manifolds}

\label{a:pdo} In this appendix, we review some basic facts on
semiclassical analysis that can be found for instance
in~\cite{Zw12}. Recall that we define on $\mathbb{R}^{2d}$ the
following class of symbols:
$$S^{m,k}(\mathbb{R}^{2d}):=\left\{(b_{\hbar}(x,\xi))_{\hbar\in(0,1]}\in \ml{C}^{\infty}(\mathbb{R}^{2d}):|\partial^{\alpha}_x\partial^{\beta}_{\xi}b_{\hbar}|
\leq
C_{\alpha,\beta}\hbar^{-k}\langle\xi\rangle^{m-|\beta|}\right\}.$$
Let $M$ be a smooth Riemannian $d$-manifold without boundary.
Consider a smooth atlas $(f_l,V_l)$ of $M$, where each $f_l$ is a
smooth diffeomorphism from $V_l\subset M$ to a bounded open set
$W_l\subset\mathbb{R}^{d}$. To each $f_l$ correspond a pull back
$f_l^*:\ml{C}^{\infty}(W_l)\rightarrow \ml{C}^{\infty}(V_l)$ and a canonical
map $\tilde{f}_l$ from $T^*V_l$ to $T^*W_l$:
$$\tilde{f}_l:(x,\xi)\mapsto\left(f_l(x),(Df_l(x)^{-1})^T\xi\right).$$
Consider now a smooth locally finite partition of identity
$(\phi_l)$ adapted to the previous atlas $(f_l,V_l)$. That means
$\sum_l\phi_l=1$ and $\phi_l\in \ml{C}^{\infty}(V_l)$. Then, any
observable $b$ in $\ml{C}^{\infty}(T^*M)$ can be decomposed as 
$b=\sum_l b_l$, where $b_l=b\phi_l$. Each $b_l$ belongs to
$\ml{C}^{\infty}(T^*V_l)$ and can be pushed to a function
$\tilde{b}_l=(\tilde{f}_l^{-1})^*b_l\in \ml{C}^{\infty}(T^*W_l)$. As
in~\cite{Zw12}, define the class of symbols of order $m$ and index
$k$
\begin{equation}
\label{e:defpdo}S^{m,k}(T^{*}M):=\left\{(b_{\hbar}(x,\xi))_{\hbar\in(0,1]}\in
\ml{C}^{\infty}(T^*M):|\partial^{\alpha}_x\partial^{\beta}_{\xi}b_{\hbar}|\leq
C_{\alpha,\beta}\hbar^{-k}\langle\xi\rangle^{m-|\beta|}\right\}.
\end{equation}
Then, for $b\in S^{m,k}(T^{*}M)$ and for each $l$, one can
associate to the symbol $\tilde{b}_l\in S^{m,k}(\mathbb{R}^{2d})$
the standard Weyl quantization
$$\Op_{\hbar}^{w}(\tilde{b}_l)u(x):=
\frac{1}{(2\pi\hbar)^d}\int_{\IR^{2d}}e^{\frac{\imath}{\hbar}\langle
x-y,\xi\rangle}\tilde{b}_l\left(\frac{x+y}{2},\xi;\hbar\right)u(y)dyd\xi,$$
where $u\in\mathcal{S}(\mathbb{R}^d)$, the Schwartz class.
Consider now a smooth cutoff $\psi_l\in \ml{C}_c^{\infty}(V_l)$ such
that $\psi_l=1$ close to the support of $\phi_l$. A quantization
of $b\in S^{m,k}(T^*M)$ is then defined in the following
way~\cite{Zw12}:
\begin{equation}
\label{e:pdomanifold}\Op_{\hbar}(b)(u):=\sum_l
\psi_l\times\left(f_l^*\Op_{\hbar}^w(\tilde{b}_l)(f_l^{-1})^*\right)\left(\psi_l\times
u\right),
\end{equation}
where $u\in \ml{C}^{\infty}(M)$. This quantization procedure
$\Op_{\hbar}$ sends (modulo $\mathcal{O}(\hbar^{\infty})$)
$S^{m,k}(T^{*}M)$ onto the space of pseudodifferential operators
of order $m$ and of index $k$, denoted
$\Psi^{m,k}(M)$~\cite{Zw12}. It can be shown that the dependence
in the cutoffs $\phi_l$ and $\psi_l$ only appears at order $1$ in
$\hbar$ (Theorem $9.10$ in~\cite{Zw12}) and the principal symbol
map $\sigma_0:\Psi^{m,k}(M)\rightarrow
S^{m,k}/S^{m-1,k-1}(T^{*}M)$ is then intrinsically defined. Most
of the rules (for example the composition of operators, the Egorov
and Calder\'on-Vaillancourt Theorems) that hold on
$\mathbb{R}^{2d}$ still hold in the case of $\Psi^{m,k}(M)$. Finally, we denote by $\Psi^{-\infty,k}(M)$ the set $\cap_{m\in\IR}\Psi^{m,k}(M)$.

\section{Time-dependent semiclassical measures}
\label{a:measures}

The aim of this short appendix is to recall a few facts on the definition of time-dependent semiclassical measures -- we refer to~\cite{Ma09} for more details.

Let $(u_{\hbar})_{0<\hbar\leq 1}$ be a normalized sequence in $L^2(M)$ verifying the oscillation assumptions~\eqref{e:hosc} and~\eqref{e:shosc}. For a given scale of times $\tau:=(\tau_{\hbar})_{\hbar\rightarrow 0^+}$ satisfying
$$\lim_{\hbar\rightarrow 0^+}\tau_{\hbar}=+\infty,$$
we denoted the Wigner distribution by
\begin{equation}\label{e:wigner-app}\forall a\in \ml{C}^{\infty}_c(T^*M),\ w_{\hbar}(t\tau_{\hbar})(a)
:=\left\la v_{\hbar}(t\tau_{\hbar}),\Oph^w(a)v_{\hbar}(t\tau_{\hbar})\right\ra,
\end{equation}
where $v_{\hbar}(\tau')$ is the solution at time $\tau'$ of~\eqref{e:perturbed-schr} with initial condition $u_{\hbar}$. Using the Calder\'on-Vaillancourt Theorem, we deduce the existence of a constant $C>0$ and a positive integer $D$ depending only on the manifold $(M,g)$ such that
$$\left|\int_{\IR\times T^*M}a(t,x,\xi) w_{\hbar}(t\tau_{\hbar},dx,d\xi)dt\right|\leq C_K \sum_{|\alpha|\leq D}\hbar^{|\alpha|/2}\int_{\IR}\|\partial_{x,\xi}^{\alpha}a(t,\cdot)\|_{\ml{C}^0(T^*M)}dt,$$
for every $a\in \ml{C}^{\infty}_c(\IR\times T^*M)$. According to~\cite{Schw66} (Ch.~$3$), the sequence $(\mu_{\hbar}:(t,x,\xi)\mapsto w_{\hbar}(t\tau_{\hbar},x,\xi))_{\hbar>0}$ is relatively compact in $\ml{D}'(\IR\times T^*M)$. Thus, we can extract converging subsequences (for the weak-$\star$ topology). In particular, for any accumulation point $\mu$ of this sequence and every $a\in \ml{C}^{\infty}_c(\IR\times T^*M)$, one has
$$\left|\int_{\IR\times T^*M}a(t,x,\xi)\mu(dt,dx,d\xi)\right|\leq C\int_{\IR}\|a(t,\cdot)\|_{\ml{C}^0(T^*M)}dt.$$
Thus, $\mu$ can be extended to a continuous linear form on $L^1(\IR,\ml{C}^0_0(T^*M))$, where $\ml{C}^0_0(T^*M)$ denotes the set of continuous functions vanishing at infinity. Consequently, the limit distribution $t\mapsto\mu(t,\cdot)$ belongs to $L^{\infty}(\IR,\ml{M}(T^*M))$ (see for instance~\cite{DieUh77}), where $\ml{M}(T^*M)$ is the set of finite complex measures on $T^*M$. For any converging subsequence in $\ml{D}'(\IR\times T^*M)$ (which we do not relabel), we note that the following also holds, for every $\theta\in L^1(\IR)$, and for every $a\in \ml{C}^{\infty}_c(T^*M)$
$$\lim_{\hbar\rightarrow 0^+}\int_{\IR\times T^*M}\theta(t)a(x,\xi) w_{\hbar}(t\tau_{\hbar},dx,d\xi)dt=\int_{\IR\times T^*M}\theta(t)a(x,\xi) \mu(t,dx,d\xi)dt.$$

Finally\footnote{One can for instance follow the arguments given in Ch.~$5$ of~\cite{Zw12}.}, according to the G\aa{}rding inequality, the limit distribution is in fact a positive measure for a.e. $t$ in $\IR$. Moreover, the frequency assumptions~\eqref{e:hosc} and~\eqref{e:shosc} and the fact that $\eps_{\hbar}\rightarrow 0^+$ imply that, for every almost every $t$, $\mu(t)(\mathring{T}^*M)=1$. Using Egorov Theorem, one can also verify that, for a.e. $t$ in $\IR$, $\mu(t,\cdot)$ is invariant by the geodesic flow $\varphi^s.$

\section{Geometry of $T^*M$}

\label{s:geom-back}

In this appendix, we collect some classical results on Riemannian
and symplectic geometry that appear at different stages of this
work. Along the way, we recall classical notations that are used
all along this article. We refer for instance the reader
to~\cite{Be78, Pa99, Rug07, Web99} for
more details.

\subsection{Musical isomorphisms}

Recall that the Riemannian metric $g$ on $M$ induces two natural
isomorphisms
$$\flat :T_xM\rightarrow T_x^*M,\ v\mapsto g_x(v,.),$$
and its inverse $\sharp: T_x^*M\rightarrow T_xM.$ This natural
isomorphism induces a positive definite form on $T_x^*M$ for which
these isomorphisms are in fact isometries. We denote by $g^*$ the
corresponding metric.

\subsection{Horizontal and vertical subbundles} Let $\rho=(x,\xi)$
be an element in $T^*M$. Denote by $\pi: T^*M\rightarrow M$ the
canonical projection $(x,\xi)\mapsto x$. We introduce the
so-called \emph{vertical} subspace:
$$\ml{V}_{\rho}:=\text{Ker}(d_{\rho}\pi)\subset T_{\rho}T^*M.$$
The fiber $T_x^*M$ is a submanifold of $T^*M$ that contains the
point $(x,\xi)$. The tangent space to this submanifold at point
$(x,\xi)$ is the vertical subspace $\ml{V}_{\rho}$ and it can be
canonically identified with $T_x^*M$. We will now define the
connection map. For that purpose, we fix $Z$ in $T_{\rho}T^*M$ and
$\rho(t)=(x(t),\xi(t))$ a smooth curve in $T^*M$ such that
$\rho(0)=\rho$ and $\rho'(0)=Z$. The connection map
$\mathcal{K}_{\rho}:T_{\rho}T^*M\mapsto T_x^*M$ is the following
application:
$$\mathcal{K}_{\rho}(Z):=\frac{D}{dt}\rho(0)=\nabla_{x'(0)}\xi(0),$$
where $\frac{D}{dt}\rho(t)$ is the covariant derivative of
$\rho(t)$ along the curve $x(t)$. One can verify that this
quantity depends only on the initial velocity $Z$ of the curve
(and not on the curve itself) and the map is linear. The
\emph{horizontal} space is given by the kernel of this linear
application, i.e.
$$\ml{H}_{\rho}:=\text{Ker}(\ml{K}_{\rho})\subset T_{\rho}T^*M.$$
There exists a natural vector bundle isomorphism between the
pullback bundle $\pi^*(TM\oplus T^*M)\rightarrow T^*M$ and the
canonical bundle $TT^*M\rightarrow T^*M$. The restriction of this
isomorphism on the fibers above $\rho\in T^*M$ is given by
$$\theta(\rho):T_{\rho} T^*M\rightarrow T_{\pi(\rho)} M\oplus T_{\pi(\rho)}^*M), \quad Z\mapsto (y,\eta):=(d_{\rho}\pi (Z),\ml{K}_{\rho}(Z)).$$
These coordinates $(y,\eta)$ will allow us to express easily the
different structures on $T^*M$. For instance, the Hamiltonian vector field $X$ associated to $p_0$ (\emph{i.e.} the generator of the geodesic flow) satisfies $\theta(\rho)X(\rho)=(\pi(\rho)^\sharp,0)$.

\subsection{Symplectic structure on $T^*M$} Recall that the
canonical contact form on $T^*M$ is given by the following
expression:
$$\forall\rho=(x,\xi)\in T^*M,\ \forall Z\in T_{\rho}T^*M,\ \alpha_{x,\xi}(Z)=\xi(d_{\rho}\pi(Z)).$$
The canonical symplectic form on $T^*M$ can then be defined as
$\Omega=d\alpha$. Using our natural isomorphism, this symplectic
form can be written as
$$\forall Z_1\cong  (y_1,\eta_1)\in T_{\rho}T^*M,\ \forall Z_2\cong  (y_2,\eta_2)\in T_{\rho}T^*M,\ \Omega_{\rho}(Z_1,Z_2)=\eta_1(y_2)-\eta_2(y_1).$$

\subsection{Almost complex structure on $T^*M$} One can define
the following map from $T_xM\oplus T_x^*M$ to itself:
$$\tilde{J}_x(y,\eta)=(\eta^{\sharp},-y^{\flat}).$$
This map induces an almost complex structure on $T_{\rho} T^*M$
through the isomorphism $\theta(\rho)$. We denote this almost
complex structure by $J_{\rho}$.

\subsection{Riemannian metric on $T^*M$}

The Sasaki metric $g^S$ on $T^*M$ is then defined as
$$g^S_{\rho}(Z_1,Z_2):=g_x^*(\ml{K}_{\rho}(Z_1),\ml{K}_{\rho}(Z_2))+g_x(d_{\rho}\pi(Z_1),d_{\rho}\pi(Z_2)).$$
This is a positive definite bilinear form on $T_{\rho}T^*M$. The
important point is that \emph{this metric is compatible with the
symplectic structure on $T^*M$ through the almost complex
structure}. Precisely, one has, for every $(Z_1,Z_2)\in
T_{\rho}T^*M\times T_{\rho} T^*M$,
$$g^S_{\rho}(Z_1,Z_2)=\Omega_{\rho}(Z_1,J_{\rho}Z_2).$$
In fact, using the natural isomorphism, one has
$$\Omega_{\rho}(Z_1,J_{\rho}Z_2)=\eta_1(\eta_2^{\sharp})+y_2^{\flat}(y_1)=g^*_x(\eta_1,\eta_2)+g_x(y_1,y_2).$$

\subsection{A natural frame on $T^*M$}\label{r:dim2}

In the case $M$ is an oriented surface, one can introduce a ``natural
basis'' on $T_{\rho} T^*M$ as follows. Thanks to the fact that the
manifold $M$ is oriented with a Riemannian structure, one can
define a notion of rotation by $\pi/2$ in every cotangent space
$T_x^*M$ (which is of dimension $2$). Thus, given any $\xi\in T_x^*M\setminus \{0\}$ there exists a
unique $\xi^{\perp}$ such that $\{\xi,\xi^{\perp}\}$ is a direct
orthogonal basis with $\|\xi\|_x=\|\xi^{\perp}\|_x$. We use this to
define an orthogonal basis of $\ml{V}_{\rho}$ for $\rho\in\mathring{T}^*M$:
$$Y_0(\rho):=\left(\theta(\rho)\right)^{-1}(0,\xi),\ \text{and}\ U(\rho)=\left(\theta(\rho)\right)^{-1}(0,\xi^{\perp}).$$
Then, we can define an orthogonal basis of $\ml{H}_{\rho}$ as
follows
$$X_0(\rho)=J_{\rho}Y_0(\rho),\ \text{and}\ W(\rho)=J_{\rho}U(\rho).$$
Note that $X_0$ is the geodesic vector field and that the family $\{X_0(\rho),W(\rho),Y_0(\rho),U(\rho)\}$ forms a direct orthogonal basis of
$T_{\rho}\mathring{T}^*M$.

\end{document}